\documentclass[a4paper,11pt,fleqn]{article}
\usepackage[obeyspaces,hyphens,spaces]{url}
\usepackage[dvipsnames]{xcolor}
\usepackage{etoolbox}
\usepackage{amsmath}
\usepackage{amssymb}
\usepackage{euscript}
\usepackage{mathrsfs}  
\usepackage{pifont}   
\usepackage{mathtools}
\usepackage{srcltx}   
\usepackage{charter}
\allowdisplaybreaks 
\definecolor{labelkey}{HTML}{0455BF}
\definecolor{refkey}{rgb}{0,0.6,0.0}
\definecolor{dblue}{HTML}{0455BF}
\definecolor{dgreen}{HTML}{02724A}
\definecolor{myellow}{HTML}{D97904}
\definecolor{dred}{HTML}{D90404}
\usepackage{upref}
\usepackage{hyperref}
\hypersetup{colorlinks=true,linktocpage=true,linkcolor=dblue,%
citecolor=dgreen,urlcolor=dred}

\renewcommand{\leq}{\ensuremath{\leqslant}}
\renewcommand{\geq}{\ensuremath{\geqslant}}

\newcommand{\Argmind}[2]{\ensuremath{%
\underset{\substack{#1}}{\text{\rm Argmin}}\;\;#2}}
\newcommand{\minmax}[3]{\ensuremath{%
\underset{\substack{#1}}{\text{\rm minimize}}\;\:
\underset{\substack{#2}}{\text{\rm maximize}}\;\;#3}}
\newcommand{\minimize}[2]{\ensuremath{\underset{\substack{{#1}}}%
{\text{\rm minimize}}\;\;#2}}
 
\newcommand{\Scal}[2]{\bigg\langle{#1}\;\bigg|\:{#2}\bigg\rangle} 
 
\newcommand{\scal}[2]{{\langle{{#1}\mid{#2}}\rangle}}
\newcommand{\sscal}[2]{{\big\langle{{#1}\mid{#2}}\big\rangle}}

\newcommand{\menge}[2]{\big\{{#1}~|~{#2}\big\}} 
\newcommand{\Menge}[2]{\left\{{#1}~\middle|~{#2}\right\}} 
\newcommand{\lag}{\ensuremath{\boldsymbol{\mathsf{A}}}}
\newcommand{\KT}{\ensuremath{\mathsf{Z}}}
\newcommand{\GGG}{\ensuremath{\boldsymbol{\mathcal{G}}}}
\newcommand{\HHH}{\ensuremath{\boldsymbol{\mathcal{H}}}}
\newcommand{\UUU}{\ensuremath{\boldsymbol{\mathcal{U}}}}
\newcommand{\VVV}{\ensuremath{\boldsymbol{\mathcal{V}}}}
\newcommand{\KKK}{\ensuremath{\boldsymbol{\mathcal{K}}}}

\newcommand{\XXX}{\ensuremath{\boldsymbol{\mathsf{X}}}}
\newcommand{\HH}{\ensuremath{{\mathcal{H}}}}

\newcommand{\GG}{\ensuremath{{\mathcal{G}}}}
\newcommand{\KK}{\ensuremath{{\mathcal{K}}}}

\newcommand{\NN}{\ensuremath{\mathbb{N}}}

\newcommand{\pnabla}[1]{\ensuremath{\nabla_{\!#1}}}
\newcommand{\Sum}{\ensuremath{\displaystyle\sum}}

\newcommand{\emp}{\ensuremath{\varnothing}}

\newcommand{\Id}{\ensuremath{\mathrm{Id}}}
\newcommand{\RR}{\ensuremath{\mathbb{R}}}
\newcommand{\RP}{\ensuremath{\left[0,{+}\infty\right[}}

\newcommand{\RPP}{\ensuremath{\left]0,{+}\infty\right[}}

\newcommand{\RPX}{\ensuremath{\left[0,{+}\infty\right]}}
\newcommand{\RX}{\ensuremath{\left]{-}\infty,{+}\infty\right]}}

\newcommand{\weakly}{\ensuremath{\rightharpoonup}}

\newcommand{\ran}{\ensuremath{\text{\rm ran}\,}}

\newcommand{\zer}{\text{\rm zer}\,}
\newcommand{\pinf}{\ensuremath{{{+}\infty}}}

\newcommand{\dom}{\ensuremath{\text{\rm dom}\,}}

\newcommand{\prox}{\ensuremath{\text{\rm prox}}}

\newcommand{\gra}{\ensuremath{\text{\rm gra}\,}}

\newcommand{\sri}{\ensuremath{\text{\rm sri}\,}}

\newcommand{\zeroun}{\ensuremath{\left]0,1\right[}}

\def\abstract{\noindent{\bfseries Abstract}. \ignorespaces}

\usepackage{theorem}
\newtheorem{theorem}{Theorem}[section]
\newtheorem{lemma}[theorem]{Lemma}

\newtheorem{proposition}[theorem]{Proposition}

\theoremstyle{plain}{\theorembodyfont{\rmfamily}%
}

\theoremstyle{plain}{\theorembodyfont{\rmfamily}%
\newtheorem{example}[theorem]{Example}}
\theoremstyle{plain}{\theorembodyfont{\rmfamily}%
\newtheorem{remark}[theorem]{Remark}}
\theoremstyle{plain}{\theorembodyfont{\rmfamily}%
}
\theoremstyle{plain}{\theorembodyfont{\rmfamily}%
}
\theoremstyle{plain}{\theorembodyfont{\rmfamily}%
\newtheorem{definition}[theorem]{Definition}}
\theoremstyle{plain}{\theorembodyfont{\rmfamily}%
\newtheorem{problem}[theorem]{Problem}}
\theoremstyle{plain}{\theorembodyfont{\rmfamily}%
}
\usepackage{enumitem}
\setlist[enumerate]{itemsep=2pt}
\setlist[itemize]{itemsep=2pt}

\numberwithin{equation}{section}
\newcommand*\mute{{\mkern 2mu\cdot\mkern 2mu}}
\usepackage{authblk}

\newcommand{\email}[1]{\href{mailto:#1}{\nolinkurl{#1}}}
\author{Minh N. B\`ui}
\author{Patrick L. Combettes}
\affil{North Carolina State University,
Department of Mathematics, Raleigh, NC 27695-8205, USA\\
\email{mnbui@ncsu.edu} and \email{plc@math.ncsu.edu}
}

\begin{document}

\title{\sffamily\huge%
A Warped Resolvent Algorithm\\ 
to Construct Nash Equilibria\thanks{%
Contact author: P. L. Combettes.
Email: \email{plc@math.ncsu.edu}.
Phone: +1 919 515 2671.
This work was supported by the National Science
Foundation under grant DMS-1818946.}
}

\date{\ttfamily ~}
\maketitle

\begin{abstract}
We propose an asynchronous block-iterative decomposition algorithm
to solve Nash equilibrium problems involving a mix of nonsmooth and
smooth functions acting on linear mixtures of strategies. The
methodology relies heavily on monotone operator theory and in
particular on warped resolvents.
\end{abstract}

\section{Introduction}
\label{sec:1}

We consider a noncooperative game with $m$ players indexed by
$I=\{1,\ldots,m\}$, in which the strategy $x_i$ of player 
$i\in I$ lies in a real Hilbert space $\HH_i$. A strategy 
profile is a point $\boldsymbol{x}=(x_i)_{i\in I}$ in the Hilbert
direct sum $\HHH=\bigoplus_{i\in I}\HH_i$, and the associated
profile of the players other than $i\in I$ is the vector
$\boldsymbol{x}_{\smallsetminus i}=(x_j)_{j\in
I\smallsetminus\{i\}}$ in 
$\HHH_{\smallsetminus i}=\bigoplus_{j\in I
\smallsetminus\{i\}}\HH_j$. Given an index $i\in I$ and a vector
$(x_i,\boldsymbol{y})\in\HH_i\times\HHH$, we set
$(x_i;\boldsymbol{y}_{\smallsetminus i})=(y_1,\ldots,y_{i-1},x_i,
y_{i+1},\ldots,y_m)$. 

A fundamental equilibrium notion was introduced by Nash in
\cite{Nash50,Nash51} to describe a state in which the loss of each
player cannot be reduced by unilateral deviation. In our context, a
formulation of the Nash equilibrium problem is
\begin{equation}
\label{e:0}
\text{find}\;\:\overline{\boldsymbol{x}}
\in\HHH\;\:\text{such that}\;\:
(\forall i\in I)\;\;\overline{x}_i\in\Argmind{x_i\in\HH_i}
{\theta_i(x_i)+\boldsymbol{\ell}_i(x_i;
\overline{\boldsymbol{x}}_{\smallsetminus i})},
\end{equation}
where the global loss function of player $i\in I$ is the sum of an
individual loss $\theta_i\colon\HH_i\to\RX$ and a joint loss
$\boldsymbol{\ell}_i\colon\HHH\to\RX$ that models the interactions
with the other players. 
Under convexity assumptions, numerical methods to solve \eqref{e:0}
have been investigated since the early 1970s \cite{Bens72} and they
have since involved increasingly sophisticated tools from nonlinear
analysis; see \cite{Atto08,Brav18,Bric13,Cohe87,Sign20,Comi12,%
Facc07,Heus09,Hoda10,Kann12,YiPa19}. In the present paper, we
consider the following highly modular Nash equilibrium problem
wherein the functions $(\theta_i)_{i\in I}$ and
$(\boldsymbol{\ell}_i)_{i\in I}$ of \eqref{e:0} are decomposed into
elementary components that are easier to process numerically.

\newpage
\begin{problem}
\label{prob:1}
Let $(\HH_i)_{i\in I}$, $(\KK_i)_{i\in I}$, and $(\GG_k)_{k\in K}$
be finite families of real Hilbert spaces, and set
$\HHH=\bigoplus_{i\in I}\HH_i$, 
$\KKK=\bigoplus_{i\in I}\KK_i$, and 
$\GGG=\bigoplus_{k\in K}\GG_k$. 
Suppose that the following are satisfied:
\begin{enumerate}[label={\rm[\alph*]}]
\item
\label{prob:1a}
For every $i\in I$, $\varphi_i\colon\HH_i\to\RX$ is proper, 
lower semicontinuous, and convex,
$\alpha_i\in\RP$, and $\psi_i\colon\HH_i\to\RR$ is convex and
differentiable with an $\alpha_i$-Lipschitzian gradient.
\item
\label{prob:1d}
For every $i\in I$,
$\boldsymbol{f}_{\!i}\colon\KKK\to\RR$ is such that, for every
$\boldsymbol{y}\in\KKK$, $\boldsymbol{f}_{\!i}
(\mute;\boldsymbol{y}_{\smallsetminus i})\colon\KK_i\to\RR$
is convex and G\^ateaux differentiable,
and we denote its gradient at $y_i\in\KK_i$ by
$\pnabla{i}\boldsymbol{f}_{\!i}(\boldsymbol{y})$.
Further, the operator $\boldsymbol{Q}\colon\KKK\to\KKK\colon
\boldsymbol{y}\mapsto
(\pnabla{i}\boldsymbol{f}_{\!i}(\boldsymbol{y}))_{i\in I}$
is monotone and Lipschitzian. Finally, $(\chi_i)_{i\in I}$ 
are positive numbers such that
\begin{equation}
\label{e:Qlip}
(\forall\boldsymbol{y}\in\KKK)(\forall\boldsymbol{y}'\in\KKK)
\quad
\scal{\boldsymbol{y}-\boldsymbol{y}'}{
\boldsymbol{Q}\boldsymbol{y}-\boldsymbol{Q}\boldsymbol{y}'}
\leq\sum_{i\in I}\chi_i\|y_i-y_i'\|^2.
\end{equation}
\item
\label{prob:1b}
For every $k\in K$, $g_k\colon\GG_k\to\RX$ is proper, lower
semicontinuous, and convex, $\beta_k\in\RP$, and 
$h_k\colon\GG_k\to\RR$ is convex and differentiable
with a $\beta_k$-Lipschitzian gradient.
\item
\label{prob:1c}
For every $i\in I$ and every $k\in K$,
$M_i\colon\HH_i\to\KK_i$ and $L_{k,i}\colon\HH_i\to\GG_k$
are linear and bounded, and, for every $\boldsymbol{x}\in\HHH$,
we write $\boldsymbol{L}_{k,\smallsetminus i}
\boldsymbol{x}_{\smallsetminus i}
=\sum_{j\in I\smallsetminus\{i\}}L_{k,j}x_j$ and
$\boldsymbol{M}\boldsymbol{x}=(M_jx_j)_{j\in I}$.
\end{enumerate}
The goal is to 
\begin{multline}
\label{e:nash}
\text{find}\;\:\overline{\boldsymbol{x}}\in\HHH\;\:
\text{such that}\;\:(\forall i\in I)\\
\overline{x}_i\in\Argmind{x_i\in\HH_i}{
\varphi_i(x_i)+\psi_i(x_i)+\boldsymbol{f}_{\!i}\big(M_ix_i;
(\boldsymbol{M}\overline{\boldsymbol{x}})_{\smallsetminus i}\big)
+\sum_{k\in K}(g_k+h_k)(L_{k,i}x_i
+\boldsymbol{L}_{k,\smallsetminus i}
\overline{\boldsymbol{x}}_{\smallsetminus i})}.
\end{multline}
\end{problem}

In Problem~\ref{prob:1}, the individual loss of player $i\in I$
consists of a nonsmooth component $\varphi_i$ and a smooth
component $\psi_i$, while his joint loss is decomposed into a
smooth function $\boldsymbol{f}_{\!i}$ and a sum of nonsmooth
functions $(g_k)_{k\in K}$ and smooth functions $(h_k)_{k\in K}$
acting on linear mixtures of the strategies. We aim at solving
\eqref{e:nash} with a numerical procedure that can be implemented
in a flexible fashion and that is able to cope with possibly very
large scale problems. This leads us to adopt the following design
principles:
\begin{itemize}
\item
{\bfseries Decomposition:}
Each function and each linear operator in
Problem~\ref{prob:1} is activated separately.
\item
{\bfseries Block-iterative implementation:}
Only a subgroup of functions needs to be
activated at any iteration. This makes it possible to best
modulate and adapt the computational load of each iteration in
large-scale problems.
\item
{\bfseries Asynchronous implementation:}
The computations are asynchronous in the sense that the result of
calculations initiated at earlier iterations can be incorporated at
the current one.
\end{itemize}

Our methodology is to first transform \eqref{e:nash} into a system
of monotone set-valued inclusions and then approach it via monotone
operator splitting techniques. Since no splitting
technique tailored to \eqref{e:nash} and compliant with the above
principles appears to be available, we adopt
a fresh perspective hinging on the theory of warped
resolvents \cite{Jmaa20}. In Section~\ref{sec:2} we provide the
necessary notation and background on monotone operator theory.
Section~\ref{sec:3} is devoted to the derivation of the proposed
asynchronous block-iterative algorithm to solve
Problem~\ref{prob:1}. Application examples are provided in 
Section~\ref{sec:4}.

\section{Notation and background}
\label{sec:2}

General background on monotone operators and related notions
can be found in \cite{Livre1}.

Let $\HH$ be a real Hilbert space. We denote by $2^{\HH}$ the power
set of $\HH$ and by $\Id$ the identity operator on $\HH$. The 
weak convergence and the strong convergence of 
a sequence $(x_n)_{n\in\NN}$ in $\HH$ to a point $x$ in $\HH$ are
denoted by $x_n\weakly x$ and $x_n\to x$, respectively.
Let $A\colon\HH\to 2^{\HH}$. The domain of $A$ is
$\dom A=\menge{x\in\HH}{Ax\neq\emp}$, the range of $A$ is
$\ran A=\bigcup_{x\in\dom A}Ax$, the graph of $A$ is
$\gra A=\menge{(x,x^*)\in\HH\times\HH}{x^*\in Ax}$, the set of 
zeros of $A$ is $\zer A=\menge{x\in\HH}{0\in Ax}$, and the 
inverse of $A$ is $A^{-1}\colon\HH\to 2^{\HH}\colon x^*\mapsto
\menge{x\in\HH}{x^*\in Ax}$. Now suppose that $A$ is monotone, 
that is, 
\begin{equation}
\big(\forall(x,x^*)\in\gra A\big)
\big(\forall(y,y^*)\in\gra A\big)\quad
\scal{x-y}{x^*-y^*}\geq 0.
\end{equation}
Then $A$ is maximally monotone if, for every
monotone operator $\widetilde{A}\colon\HH\to 2^{\HH}$,
$\gra A\subset\gra\widetilde{A}$ $\Rightarrow$
$A=\widetilde{A}$; $A$ is strongly monotone with constant
$\varsigma\in\RPP$ if $A-\varsigma\Id$ is monotone; and
$A$ is $3^*$ monotone if 
\begin{equation}
(\forall x\in\dom A)(\forall x^*\in\ran A)\quad
\sup_{(y,y^*)\in\gra A}\scal{x-y}{y^*-x^*}<\pinf.
\end{equation}
$\Gamma_0(\HH)$ is the set of lower semicontinuous convex
functions $\varphi\colon\HH\to\RX$ which are proper in the sense
that $\dom\varphi=\menge{x\in\HH}{\varphi(x)<\pinf}\neq\emp$.
Let $\varphi\in\Gamma_0(\HH)$.
Then $\varphi$ is supercoercive if
$\lim_{\|x\|\to\pinf}\varphi(x)/\|x\|=\pinf$
and uniformly convex if there exists an increasing function
$\phi\colon\RP\to\RPX$ that vanishes only at $0$ such that
\begin{multline}
(\forall x\in\dom\varphi)(\forall y\in\dom\varphi)
(\forall\alpha\in\zeroun)\\
\varphi\big(\alpha x+(1-\alpha)y\big)+\alpha(1-\alpha)
\phi\big(\|x-y\|\big)\leq\alpha\varphi(x)
+(1-\alpha)\varphi(y).
\end{multline}
For every $x\in\HH$, $\prox_\varphi x$ denotes
the unique minimizer of $\varphi+(1/2)\|\mute -x\|^2$.
The subdifferential of $\varphi$
is the maximally monotone operator
$\partial\varphi\colon\HH\to 2^{\HH}\colon
x\mapsto\menge{x^*\in\HH}{(\forall y\in\HH)\;
\scal{y-x}{x^*}+\varphi(x)\leq\varphi(y)}$.
Let $C$ be a convex subset of $\HH$.
The indicator function of $C$ is
\begin{equation}
\iota_C\colon\HH\to\RPX\colon x\mapsto
\begin{cases}
0,&\text{if}\;\:x\in C;\\
\pinf,&\text{otherwise},
\end{cases}
\end{equation}
and the strong relative interior of $C$ is
\begin{equation}
\sri C=\Menge{x\in C}{\bigcup_{\lambda\in\,\RPP}\lambda(C-x)
\;\:\text{is a closed vector subspace of}\;\:\HH}.
\end{equation}

The following notion of a warped resolvent will be instrumental to
our approach.

\begin{definition}[\cite{Jmaa20}]
\label{d:warped}
Suppose that $\XXX$ is a real Hilbert space.
Let $\boldsymbol{\mathsf{D}}$ be a nonempty subset of $\XXX$,
let $\boldsymbol{\mathsf{K}}\colon
\boldsymbol{\mathsf{D}}\to\XXX$, and let
$\lag\colon\XXX\to 2^{\XXX}$ be such that
$\ran\boldsymbol{\mathsf{K}}\subset
\ran\!(\boldsymbol{\mathsf{K}}+\lag)$ and
$\boldsymbol{\mathsf{K}}+\lag$ is injective.
The \emph{warped resolvent} of $\lag$ with kernel
$\boldsymbol{\mathsf{K}}$ is
$J_{\lag}^{\boldsymbol{\mathsf{K}}}=
(\boldsymbol{\mathsf{K}}+\lag)^{-1}\circ\boldsymbol{\mathsf{K}}$.
\end{definition}

We now provide a warped resolvent algorithm to find a zero of a
maximally monotone operator $\lag\colon\XXX\to 2^{\XXX}$,
where $\XXX$ is a real Hilbert space.
This algorithm has a simple geometric interpretation:
at iteration $n$, we use the evaluation of the
warped resolvent $J_{\lag}^{\boldsymbol{\mathsf{K}}_n}$ at a
perturbation $\widetilde{\boldsymbol{\mathsf{x}}}_n$ of the current
iterate $\boldsymbol{\mathsf{x}}_n$ to construct a point
$(\boldsymbol{\mathsf{y}}_n,\boldsymbol{\mathsf{y}}_n^*)\in
\gra\lag$. By monotonicity of $\lag$, 
\begin{equation}
\label{e:d8}
\zer\lag\subset\boldsymbol{\mathsf{H}}_n=
\menge{\boldsymbol{\mathsf{z}}\in\XXX}
{\scal{\boldsymbol{\mathsf{z}}-\boldsymbol{\mathsf{y}}_n}
{\boldsymbol{\mathsf{y}}_n^*}\leq 0},
\end{equation}
and the update $\boldsymbol{\mathsf{x}}_{n+1}$ is a relaxed
projection of $\boldsymbol{\mathsf{x}}_n$ onto the half-space
$\boldsymbol{\mathsf{H}}_n$.

\begin{proposition}
\label{p:5}
Let $\XXX$ be a real Hilbert space and let
$\lag\colon\XXX\to 2^{\XXX}$ be a maximally
monotone operator such that $\zer\lag\neq\emp$.
Let $\boldsymbol{\mathsf{x}}_0\in\XXX$,
let $\varepsilon\in\zeroun$,
let $\varsigma\in\RPP$,
and let $\varpi\in\left]\varsigma,\pinf\right[$.
Further, for every $n\in\NN$,
let $\lambda_n\in\left[\varepsilon,2-\varepsilon\right]$,
let $\widetilde{\boldsymbol{\mathsf{x}}}_n\in\XXX$, and
let $\boldsymbol{\mathsf{K}}_n\colon\XXX\to\XXX$ be a 
$\varsigma$-strongly monotone and $\varpi$-Lipschitzian operator.
Iterate
\begin{equation}
\label{e:8824}
\begin{array}{l}
\text{for}\;n=0,1,\ldots\\
\left\lfloor
\begin{array}{l}
\boldsymbol{\mathsf{y}}_n
=J_{\lag}^{\boldsymbol{\mathsf{K}}_n}
\widetilde{\boldsymbol{\mathsf{x}}}_n;\\
\boldsymbol{\mathsf{y}}_n^*
=\boldsymbol{\mathsf{K}}_n\widetilde{\boldsymbol{\mathsf{x}}}_n
-\boldsymbol{\mathsf{K}}_n\boldsymbol{\mathsf{y}}_n;\\
\text{if}\;\scal{\boldsymbol{\mathsf{y}}_n
-\boldsymbol{\mathsf{x}}_n}{\boldsymbol{\mathsf{y}}_n^*}<0\\
\left\lfloor
\begin{array}{l}
\boldsymbol{\mathsf{x}}_{n+1}
=\boldsymbol{\mathsf{x}}_n
+\dfrac{\lambda_n\scal{\boldsymbol{\mathsf{y}}_n
-\boldsymbol{\mathsf{x}}_n}{\boldsymbol{\mathsf{y}}_n^*}}
{\|\boldsymbol{\mathsf{y}}_n^*\|^2}\,\boldsymbol{\mathsf{y}}_n^*;\\
\end{array}
\right.\\
\text{else}\\
\left\lfloor
\begin{array}{l}
\boldsymbol{\mathsf{x}}_{n+1}=\boldsymbol{\mathsf{x}}_n.
\end{array}
\right.\\[2mm]
\end{array}
\right.\\
\end{array}
\end{equation}
Then the following hold:
\begin{enumerate}
\item
\label{p:5i}
$\sum_{n\in\NN}\|\boldsymbol{\mathsf{x}}_{n+1}
-\boldsymbol{\mathsf{x}}_n\|^2<\pinf$.
\item
\label{p:5ii}
Suppose that $\widetilde{\boldsymbol{\mathsf{x}}}_n
-\boldsymbol{\mathsf{x}}_n\to\boldsymbol{\mathsf{0}}$.
Then $\boldsymbol{\mathsf{x}}_n-\boldsymbol{\mathsf{y}}_n
\to\boldsymbol{\mathsf{0}}$ and 
$(\boldsymbol{\mathsf{x}}_n)_{n\in\NN}$
converges weakly to a point in $\zer\lag$.
\end{enumerate}
\end{proposition}
\begin{proof}
It follows from \cite[Proposition~3.9(i){[d]}\&(ii){[b]}]{Jmaa20}
that the warped resolvents
$(J_{\lag}^{\boldsymbol{\mathsf{K}}_n})_{n\in\NN}$ in
\eqref{e:8824} are well defined. In turn, we derive 
\ref{p:5i} and the weak convergence claim from 
\cite[Theorem~4.2 and Remark~4.3]{Jmaa20}. It thus remains to prove
that $\boldsymbol{\mathsf{x}}_n-\boldsymbol{\mathsf{y}}_n
\to\boldsymbol{\mathsf{0}}$. It is shown in the proof of 
\cite[Theorem~4.2(ii)]{Jmaa20} that
$\boldsymbol{\mathsf{K}}_n\widetilde{\boldsymbol{\mathsf{x}}}_n
-\boldsymbol{\mathsf{K}}_n\boldsymbol{\mathsf{y}}_n
\to\boldsymbol{\mathsf{0}}$. At the same time,
for every $n\in\NN$, every $\boldsymbol{\mathsf{x}}\in\XXX$,
and every $\boldsymbol{\mathsf{y}}\in\XXX$,
we deduce from the Cauchy--Schwarz inequality that
$\varsigma\|\boldsymbol{\mathsf{x}}-\boldsymbol{\mathsf{y}}\|^2
\leq\scal{\boldsymbol{\mathsf{x}}-\boldsymbol{\mathsf{y}}}{
\boldsymbol{\mathsf{K}}_n\boldsymbol{\mathsf{x}}-
\boldsymbol{\mathsf{K}}_n\boldsymbol{\mathsf{y}}}
\leq\|\boldsymbol{\mathsf{x}}-\boldsymbol{\mathsf{y}}\|\,
\|\boldsymbol{\mathsf{K}}_n\boldsymbol{\mathsf{x}}-
\boldsymbol{\mathsf{K}}_n\boldsymbol{\mathsf{y}}\|$,
from which it follows that
\begin{equation}
\label{e:1878}
\varsigma\|\boldsymbol{\mathsf{x}}-\boldsymbol{\mathsf{y}}\|
\leq\|\boldsymbol{\mathsf{K}}_n\boldsymbol{\mathsf{x}}-
\boldsymbol{\mathsf{K}}_n\boldsymbol{\mathsf{y}}\|.
\end{equation}
Therefore, $\|\boldsymbol{\mathsf{x}}_n-\boldsymbol{\mathsf{y}}_n\|
\leq\|\boldsymbol{\mathsf{x}}_n
-\widetilde{\boldsymbol{\mathsf{x}}}_n\|
+\|\widetilde{\boldsymbol{\mathsf{x}}}_n
-\boldsymbol{\mathsf{y}}_n\|
\leq\|\boldsymbol{\mathsf{x}}_n
-\widetilde{\boldsymbol{\mathsf{x}}}_n\|
+(1/\varsigma)
\|\boldsymbol{\mathsf{K}}_n\widetilde{\boldsymbol{\mathsf{x}}}_n
-\boldsymbol{\mathsf{K}}_n\boldsymbol{\mathsf{y}}_n\|
\to 0$, as desired.
\end{proof}

\section{Algorithm}
\label{sec:3}

As mentioned in Section~\ref{sec:1}, there exists no method
tailored to the format of Problem~\ref{prob:1} that can solve it in
an asynchronous block-iterative fashion. Our methodology to design
such an algorithm can be broken down in the following steps:
\begin{enumerate}[label={\rm{\bfseries\arabic*.}}]
\item
We rephrase \eqref{e:nash} as a monotone inclusion problem in
$\HHH$, namely, 
\begin{equation}
\label{e:50}
\text{find}\;\:\overline{\boldsymbol{x}}\in\HHH\;\:
\text{such that}\;\:\boldsymbol{0}\in\boldsymbol{A}
\overline{\boldsymbol{x}}
+\boldsymbol{M}^*\big(\boldsymbol{Q}(\boldsymbol{M}
\overline{\boldsymbol{x}})\big)
+\boldsymbol{L}^*\big(\boldsymbol{B}(\boldsymbol{L}
\overline{\boldsymbol{x}})\big),
\end{equation}
where $\boldsymbol{Q}$ and $\boldsymbol{M}$ are defined in
Problem~\ref{prob:1}\ref{prob:1d} and
Problem~\ref{prob:1}\ref{prob:1c}, respectively, and
\begin{equation}
\label{e:3851}
\begin{cases}
\displaystyle
\boldsymbol{A}\colon\HHH\to 2^{\HHH}\colon\boldsymbol{x}\mapsto
\bigtimes_{i\in I}\big(\partial\varphi_i(x_i)
+\nabla\psi_i(x_i)\big)\\
\displaystyle
\boldsymbol{B}\colon\GGG\to 2^{\GGG}\colon
\boldsymbol{z}\mapsto\bigtimes_{k\in K}\big(\partial g_k(z_k)
+\nabla h_k(z_k)\big)\\
\boldsymbol{L}\colon\HHH\to\GGG\colon\boldsymbol{x}\mapsto
\big(\sum_{i\in I}L_{k,i}x_i\big)_{k\in K}.
\end{cases}
\end{equation}
\item
The inclusion in \eqref{e:50} involves more than two operators,
namely $\boldsymbol{A}$, $\boldsymbol{B}$, $\boldsymbol{Q}$,
$\boldsymbol{L}$, and $\boldsymbol{M}$. Hence, in the
spirit of the decomposition methodologies of
\cite{Jmaa20,Siop13,MaPr18}, a space bigger than $\HHH$ is
required to devise a splitting method to solve it. We set
$\XXX=\HHH\oplus\KKK\oplus\GGG\oplus\KKK\oplus\GGG$ and consider
the inclusion problem
\begin{equation}
\label{e:51}
\text{find}\;\:\overline{\boldsymbol{\mathsf{x}}}\in\XXX\;\:
\text{such that}\;\:\boldsymbol{\mathsf{0}}\in
\lag\overline{\boldsymbol{\mathsf{x}}},
\end{equation}
where
\begin{multline}
\label{e:lag}
\lag\colon\XXX\to 2^{\XXX}\colon
(\boldsymbol{x},\boldsymbol{y},\boldsymbol{z},
\boldsymbol{u}^*,\boldsymbol{v}^*)\mapsto\\
(\boldsymbol{A}\boldsymbol{x}+\boldsymbol{M}^*\boldsymbol{u}^*
+\boldsymbol{L}^*\boldsymbol{v}^*)
\times\{\boldsymbol{Q}\boldsymbol{y}-\boldsymbol{u}^*\}
\times(\boldsymbol{B}\boldsymbol{z}-\boldsymbol{v}^*)
\times\{\boldsymbol{y}-\boldsymbol{M}\boldsymbol{x}\}
\times\{\boldsymbol{z}-\boldsymbol{L}\boldsymbol{x}\}.
\end{multline}
\item
We show that, if 
$\boldsymbol{\mathsf{x}}=(\boldsymbol{x},\boldsymbol{y},
\boldsymbol{z},\boldsymbol{u}^*,\boldsymbol{v}^*)$ solves 
\eqref{e:51}, then $\boldsymbol{x}$ solves \eqref{e:50} and,
therefore, \eqref{e:nash}.
\item[\bfseries 4.]
To solve \eqref{e:51}, we implement the warped resolvent algorithm
of Proposition~\ref{p:5} with a specific choice of the
auxiliary points 
$(\widetilde{\boldsymbol{\mathsf{x}}}_n)_{n\in\NN}$ and the kernels
$(\boldsymbol{\mathsf{K}}_n)_{n\in\NN}$ that will lead to an
asynchronous block-iterative splitting algorithm.
\end{enumerate}

The methodology just described is put in motion in our main
theorem, which we now state and prove. 

\begin{theorem}
\label{t:1}
Consider the setting of Problem~\ref{prob:1}.
Let $\eta\in\RPP$ and $\varepsilon\in\zeroun$ be such that
$1/\varepsilon>\max\{\alpha_i+\eta,\beta_k+\eta,
\chi_i+\eta\}_{i\in I,k\in K}$, let $(\lambda_n)_{n\in\NN}$ be in
$\left[\varepsilon,2-\varepsilon\right]$, and let $D\in\NN$.
Suppose that the following are satisfied:
\begin{enumerate}[label={\rm[\alph*]}]
\item
\label{t:1a}
For every $i\in I$ and every $n\in\NN$, $\tau_i(n)\in\NN$ 
satisfies $n-D\leq\tau_i(n)\leq n$,
$\gamma_{i,n}\in\left[\varepsilon,1/(\alpha_i+\eta)\right]$,
$\mu_{i,n}\in\left[\varepsilon,1/(\chi_i+\eta)\right]$,
$\sigma_{i,n}\in\left[\varepsilon,1/\varepsilon\right]$,
$x_{i,0}\in\HH_i$, $y_{i,0}\in\KK_i$, and $u_{i,0}^*\in\KK_i$.
\item
\label{t:1b}
For every $k\in K$ and every $n\in\NN$,
$\delta_k(n)\in\NN$ satisfies $n-D\leq\delta_k(n)\leq n$,
$\nu_{k,n}\in\left[\varepsilon,1/(\beta_k+\eta)\right]$,
$\varrho_{k,n}\in\left[\varepsilon,1/\varepsilon\right]$,
$z_{k,0}\in\GG_k$, and $v_{k,0}^*\in\GG_k$.
\item
\label{t:1c}
$(I_n)_{n\in\NN}$ are nonempty subsets of $I$
and $(K_n)_{n\in\NN}$ are nonempty subsets of $K$ such that,
for some $P\in\NN$,
\begin{equation}
\label{e:1093}
I_0=I,\quad K_0=K,\quad\text{and}\quad(\forall n\in\NN)\;\;
\bigcup_{j=n}^{n+P}I_j=I\;\:\text{and}\;\:
\bigcup_{j=n}^{n+P}K_j=K.
\end{equation}
\end{enumerate}
Iterate
\begin{equation}
\label{e:dai}
\begin{array}{l}
\text{for}\;n=0,1,\ldots\\
\left\lfloor
\begin{array}{l}
\text{for every}\;i\in I_n\\
\left\lfloor
\begin{array}{l}
q_{i,n}=y_{i,\tau_i(n)}+\mu_{i,\tau_i(n)}\big(u_{i,\tau_i(n)}^*
-\pnabla{i}\boldsymbol{f}_{\!i}\big(\boldsymbol{y}_{\tau_i(n)}\big)
\big);\\
c_{i,n}^*=u_{i,\tau_i(n)}^*
+\sigma_{i,\tau_i(n)}\big(M_ix_{i,\tau_i(n)}-y_{i,\tau_i(n)}\big);
\\
x_{i,n}^*=x_{i,\tau_i(n)}-\gamma_{i,\tau_i(n)}\big(
\nabla\psi_i\big(x_{i,\tau_i(n)}\big)+M_i^*u_{i,\tau_i(n)}^*
+\sum_{k\in K}L_{k,i}^*v_{k,\tau_i(n)}^*\big);\\
a_{i,n}=\prox_{\gamma_{i,\tau_i(n)}\varphi_i}x_{i,n}^*;\\
s_{i,n}^*=\gamma_{i,\tau_i(n)}^{-1}(x_{i,n}^*-a_{i,n})
+\nabla\psi_i(a_{i,n})+M_i^*c_{i,n}^*;\\
c_{i,n}=q_{i,n}-M_ia_{i,n};
\end{array}
\right.
\\
\text{for every}\;i\in I\smallsetminus I_n\\
\left\lfloor
\begin{array}{l}
q_{i,n}=q_{i,n-1};\;c_{i,n}^*=c_{i,n-1}^*;\;
a_{i,n}=a_{i,n-1};\;s_{i,n}^*=s_{i,n-1}^*;\;
c_{i,n}=c_{i,n-1};
\end{array}
\right.\\
\text{for every}\;k\in K_n\\
\left\lfloor
\begin{array}{l}
d_{k,n}^*=z_{k,\delta_k(n)}+\nu_{k,\delta_k(n)}\big(
v_{k,\delta_k(n)}^*
-\nabla h_k\big(z_{k,\delta_k(n)}\big)\big);\\
b_{k,n}=\prox_{\nu_{k,\delta_k(n)}g_k}d_{k,n}^*;\\
e_{k,n}^*=v_{k,\delta_k(n)}^*+\varrho_{k,\delta_k(n)}\big(
\sum_{i\in I}L_{k,i}x_{i,\delta_k(n)}-z_{k,\delta_k(n)}\big);\\
b_{k,n}^*=\nu_{k,\delta_k(n)}^{-1}(d_{k,n}^*-b_{k,n})
+\nabla h_k(b_{k,n})-e_{k,n}^*;\\
e_{k,n}=b_{k,n}-\sum_{i\in I}L_{k,i}a_{i,n};
\end{array}
\right.\\
\text{for every}\;k\in K\smallsetminus K_n\\
\left\lfloor
\begin{array}{l}
b_{k,n}=b_{k,n-1};\;e_{k,n}^*=e_{k,n-1}^*;\;
b_{k,n}^*=b_{k,n-1}^*;\;
e_{k,n}=b_{k,n}-\sum_{i\in I}L_{k,i}a_{i,n};
\end{array}
\right.\\
\text{for every}\;i\in I\\
\left\lfloor
\begin{array}{l}
a_{i,n}^*=s_{i,n}^*+\sum_{k\in K}L_{k,i}^*e_{k,n}^*;\\
q_{i,n}^*=\pnabla{i}\boldsymbol{f}_{\!i}(\boldsymbol{q}_n)
-c_{i,n}^*;
\end{array}
\right.\\
\begin{aligned}
\pi_n&=\textstyle\sum_{i\in I}\big(
\scal{a_{i,n}-x_{i,n}}{a_{i,n}^*}
+\scal{q_{i,n}-y_{i,n}}{q_{i,n}^*}
+\scal{c_{i,n}}{c_{i,n}^*-u_{i,n}^*}\big)\\
&\textstyle
\quad\;+\sum_{k\in K}\big(\scal{b_{k,n}-z_{k,n}}{b_{k,n}^*}
+\scal{e_{k,n}}{e_{k,n}^*-v_{k,n}^*}\big);
\end{aligned}
\\
\text{if}\;\pi_n<0\\
\left\lfloor
\begin{array}{l}
\theta_n=\lambda_n\pi_n/\big(
\sum_{i\in I}\big(\|a_{i,n}^*\|^2+\|q_{i,n}^*\|^2
+\|c_{i,n}\|^2\big)
+\sum_{k\in K}\big(\|b_{k,n}^*\|^2+\|e_{k,n}\|^2\big)\big);\\
\text{for every}\;i\in I\\
\left\lfloor
\begin{array}{l}
x_{i,n+1}=x_{i,n}+\theta_na_{i,n}^*;\;
y_{i,n+1}=y_{i,n}+\theta_nq_{i,n}^*;\;
u_{i,n+1}^*=u_{i,n}^*+\theta_nc_{i,n};
\end{array}
\right.\\
\text{for every}\;k\in K\\
\left\lfloor
\begin{array}{l}
z_{k,n+1}=z_{k,n}+\theta_nb_{k,n}^*;\;
v_{k,n+1}^*=v_{k,n}^*+\theta_ne_{k,n};
\end{array}
\right.\\[2mm]
\end{array}
\right.\\
\text{else}\\
\left\lfloor
\begin{array}{l}
\text{for every}\;i\in I\\
\left\lfloor
\begin{array}{l}
x_{i,n+1}=x_{i,n};\;
y_{i,n+1}=y_{i,n};\;
u_{i,n+1}^*=u_{i,n}^*;
\end{array}
\right.\\
\text{for every}\;k\in K\\
\left\lfloor
\begin{array}{l}
z_{k,n+1}=z_{k,n};\;
v_{k,n+1}^*=v_{k,n}^*.
\end{array}
\right.\\[2mm]
\end{array}
\right.\\[10mm]
\end{array}
\right.
\end{array}
\end{equation}
Furthermore, suppose that there exist
$\widehat{\boldsymbol{x}}\in\HHH$,
$\widehat{\boldsymbol{u}}^*\in\KKK$, and
$\widehat{\boldsymbol{v}}^*\in\GGG$ such that
\begin{equation}
\label{e:7592}
\begin{cases}
(\forall i\in I)\;\;\widehat{u}_i^*
=\pnabla{i}\boldsymbol{f}_{\!i}
(\boldsymbol{M}\widehat{\boldsymbol{x}})\\
(\forall k\in K)\;\;\widehat{v}_k^*\in
(\partial g_k+\nabla h_k)\big(
\sum_{j\in I}L_{k,j}\widehat{x}_j\big)\\
(\forall i\in I)\;\;{-}M_i^*\widehat{u}_i^*
-\sum_{k\in K}L_{k,i}^*\widehat{v}_k^*\in
\partial\varphi_i(\widehat{x}_i)+\nabla\psi_i(\widehat{x}_i).
\end{cases}
\end{equation}
Then $({\boldsymbol{x}}_n)_{n\in\NN}$ converges weakly to a
solution to Problem~\ref{prob:1}. 
\end{theorem}
\begin{proof}
Set $\XXX=\HHH\oplus\KKK\oplus\GGG\oplus\KKK\oplus\GGG$ and
consider the operators defined in \eqref{e:3851} and \eqref{e:lag}.
Let us first examine some properties of the operator
$\lag$ in \eqref{e:lag}. For every $i\in I$, it results from
Problem~\ref{prob:1}\ref{prob:1a} and
\cite[Theorem~20.25 and Proposition~17.31(i)]{Livre1} that
$\partial\varphi_i$ and $\nabla\psi_i$
are maximally monotone and, therefore, from
\cite[Corollary~25.5(i)]{Livre1} that
$\partial\varphi_i+\nabla\psi_i$ is maximally monotone.
Thus, in view of \eqref{e:3851} and
\cite[Proposition~20.23]{Livre1},
$\boldsymbol{A}$ is maximally monotone.
Likewise, $\boldsymbol{B}$ is maximally monotone.
Hence, since $\boldsymbol{Q}$ is maximally monotone by virtue of
Problem~\ref{prob:1}\ref{prob:1d} and
\cite[Corollary~20.28]{Livre1}, \cite[Proposition~20.23]{Livre1} 
implies that the operator
\begin{equation}
\label{e:6615}
\boldsymbol{\mathsf{R}}\colon\XXX\to 2^{\XXX}\colon
(\boldsymbol{x},\boldsymbol{y},\boldsymbol{z},
\boldsymbol{u}^*,\boldsymbol{v}^*)\mapsto
\boldsymbol{A}\boldsymbol{x}
\times\{\boldsymbol{Q}\boldsymbol{y}\}
\times\boldsymbol{B}\boldsymbol{z}
\times\{\boldsymbol{0}\}
\times\{\boldsymbol{0}\}
\end{equation}
is maximally monotone. On the other hand, since the operator
\begin{equation}
\label{e:7335}
\boldsymbol{\mathsf{S}}\colon\XXX\to\XXX\colon
(\boldsymbol{x},\boldsymbol{y},\boldsymbol{z},
\boldsymbol{u}^*,\boldsymbol{v}^*)\mapsto
(\boldsymbol{M}^*\boldsymbol{u}^*+\boldsymbol{L}^*\boldsymbol{v}^*)
\times\{{-}\boldsymbol{u}^*\}
\times\{{-}\boldsymbol{v}^*\}
\times\{\boldsymbol{y}-\boldsymbol{M}\boldsymbol{x}\}
\times\{\boldsymbol{z}-\boldsymbol{L}\boldsymbol{x}\}
\end{equation}
is linear and bounded with
\begin{equation}
\label{e:2957}
\boldsymbol{\mathsf{S}}^*={-}\boldsymbol{\mathsf{S}},
\end{equation}
we deduce from \cite[Example~20.35]{Livre1} that
$\boldsymbol{\mathsf{S}}$ is maximally monotone.
In turn, it follows from \eqref{e:lag}, \eqref{e:6615},
and \cite[Corollary~25.5(i)]{Livre1} that
\begin{equation}
\label{e:6873}
\lag=\boldsymbol{\mathsf{R}}+\boldsymbol{\mathsf{S}}
\;\text{is maximally monotone}.
\end{equation}
Upon setting
$\widehat{\boldsymbol{y}}
=\boldsymbol{M}\widehat{\boldsymbol{x}}$ and
$\widehat{\boldsymbol{z}}=\boldsymbol{L}\widehat{\boldsymbol{x}}$,
we derive from \eqref{e:7592} and \eqref{e:3851} that
$\widehat{\boldsymbol{u}}^*
=\boldsymbol{Q}\widehat{\boldsymbol{y}}$ and
$\widehat{\boldsymbol{v}}^*\in
\boldsymbol{B}\widehat{\boldsymbol{z}}$.
Further, since
\begin{equation}
\label{e:8792}
\boldsymbol{M}^*\colon\KKK\to\HHH\colon\boldsymbol{u}^*
\mapsto(M_i^*u_i^*)_{i\in I}\quad\text{and}\quad
\boldsymbol{L}^*\colon\GGG\to\HHH\colon
\boldsymbol{v}^*\mapsto\Bigg(\sum_{k\in K}
L_{k,i}^*v_k^*\Bigg)_{i\in I},
\end{equation}
it results from \eqref{e:7592} and \eqref{e:3851} that
${-}\boldsymbol{M}^*\widehat{\boldsymbol{u}}^*
{-}\boldsymbol{L}^*\widehat{\boldsymbol{v}}^*\in
\boldsymbol{A}\widehat{\boldsymbol{x}}$.
Therefore, we infer from \eqref{e:lag} that
$(\widehat{\boldsymbol{x}},\widehat{\boldsymbol{y}},
\widehat{\boldsymbol{z}},\widehat{\boldsymbol{u}}^*,
\widehat{\boldsymbol{v}}^*)\in\zer\lag$ and, hence, that
\begin{equation}
\label{e:zerL}
\zer\lag\neq\emp.
\end{equation}
Define
\begin{equation}
\label{e:4469a}
(\forall i\in I)(\forall n\in\NN)\quad
\overline{\ell}_i(n)
=\max\!\menge{j\in\NN}{j\leq n\;\:\text{and}\;\:i\in I_j}
\quad\text{and}\quad
\ell_i(n)=\tau_i\big(\overline{\ell}_i(n)\big)
\end{equation}
and
\begin{equation}
\label{e:4469b}
(\forall k\in K)(\forall n\in\NN)\quad
\overline{\vartheta}_k(n)
=\max\!\menge{j\in\NN}{j\leq n\;\:\text{and}\;\:k\in K_j}
\quad\text{and}\quad
\vartheta_k(n)=\delta_k\big(\overline{\vartheta}_k(n)\big).
\end{equation}
In addition, let $\kappa\in\RPP$ be a Lipschitz constant of 
$\boldsymbol{Q}$ in Problem~\ref{prob:1}\ref{prob:1d}, set
\begin{equation}
\label{e:1688}
\begin{cases}
\alpha=\sqrt{2\big(\varepsilon^{-2}
+\max_{i\in I}\alpha_i^2\big)},\;\:
\beta=\sqrt{2\big(\varepsilon^{-2}+\max_{k\in K}\beta_k^2\big)},
\;\:\chi=\sqrt{2(\varepsilon^{-2}+\kappa^2)}\\
\varsigma=\min\{\varepsilon,\eta\},\,\:
\varpi=\|\boldsymbol{\mathsf{S}}\|
+\max\{\alpha,\beta,\chi,1/\varepsilon\},
\end{cases}
\end{equation}
and define
\begin{equation}
\label{e:1372a}
(\forall n\in\NN)\quad
\begin{cases}
\boldsymbol{E}_n\colon\HHH\to\HHH\colon\boldsymbol{x}\mapsto
\big(\gamma_{i,\ell_i(n)}^{-1}x_i
-\nabla\psi_i(x_i)\big)_{i\in I}\\
\boldsymbol{F}_{\!n}\colon\KKK\to\KKK\colon\boldsymbol{y}\mapsto
\big(\mu_{i,\ell_i(n)}^{-1}y_i
-\pnabla{i}\boldsymbol{f}_{\!i}(\boldsymbol{y})\big)_{i\in I}\\
\boldsymbol{G}_n\colon\GGG\to\GGG\colon\boldsymbol{z}\mapsto
\big(\nu_{k,\vartheta_k(n)}^{-1}z_k-\nabla h_k(z_k)\big)_{k\in K}
\\
\boldsymbol{\mathsf{T}}_n\colon\XXX\to\XXX\colon
(\boldsymbol{x},\boldsymbol{y},\boldsymbol{z},
\boldsymbol{u}^*,\boldsymbol{v}^*)\mapsto\Big(
\boldsymbol{E}_n\boldsymbol{x},
\boldsymbol{F}_{\!n}\boldsymbol{y},
\boldsymbol{G}_n\boldsymbol{z},
\big(\sigma_{i,\ell_i(n)}^{-1}u_i^*\big)_{i\in I},
\big(\varrho_{k,\vartheta_k(n)}^{-1}v_k^*\big)_{k\in K}\Big)\\
\boldsymbol{\mathsf{K}}_n
=\boldsymbol{\mathsf{T}}_n-\boldsymbol{\mathsf{S}}.
\end{cases}
\end{equation}
Fix temporarily $n\in\NN$.
Then, using \ref{t:1a}, the Cauchy--Schwarz inequality,
and Problem~\ref{prob:1}\ref{prob:1a}, we obtain
\begin{align}
&(\forall\boldsymbol{x}\in\HHH)(\forall\boldsymbol{x}'\in\HHH)
\quad\scal{\boldsymbol{x}-\boldsymbol{x}'}{
\boldsymbol{E}_n\boldsymbol{x}-\boldsymbol{E}_n\boldsymbol{x}'}
\nonumber\\
&\hspace{42mm}
=\sum_{i\in I}\big(\gamma_{i,\ell_i(n)}^{-1}\|x_i-x_i'\|^2
-\scal{x_i-x_i'}{\nabla\psi_i(x_i)-\nabla\psi_i(x_i')}\big)
\nonumber\\
&\hspace{42mm}
\geq\sum_{i\in I}\big((\alpha_i+\eta)\|x_i-x_i'\|^2
-\|x_i-x_i'\|\,\|\nabla\psi_i(x_i)-\nabla\psi_i(x_i')\|\big)
\nonumber\\
&\hspace{42mm}
\geq\sum_{i\in I}\big((\alpha_i+\eta)\|x_i-x_i'\|^2
-\alpha_i\|x_i-x_i'\|^2\big)
\nonumber\\
&\hspace{42mm}
=\eta\|\boldsymbol{x}-\boldsymbol{x}'\|^2
\end{align}
and
\begin{align}
&(\forall\boldsymbol{x}\in\HHH)(\forall\boldsymbol{x}'\in\HHH)
\quad\|\boldsymbol{E}_n\boldsymbol{x}
-\boldsymbol{E}_n\boldsymbol{x}'\|^2
\nonumber\\
&\hspace{42mm}
=\sum_{i\in I}\big\|\gamma_{i,\ell_i(n)}^{-1}(x_i-x_i')
-\big(\nabla\psi_i(x_i)-\nabla\psi_i(x_i')\big)\big\|^2
\nonumber\\
&\hspace{42mm}
\leq 2\sum_{i\in I}\big(\gamma_{i,\ell_i(n)}^{-2}\|x_i-x_i'\|^2
+\|\nabla\psi_i(x_i)-\nabla\psi_i(x_i')\|^2\big)
\nonumber\\
&\hspace{42mm}
\leq 2\sum_{i\in I}\big(\varepsilon^{-2}\|x_i-x_i'\|^2
+\alpha_i^2\|x_i-x_i'\|^2\big)
\nonumber\\
&\hspace{42mm}
\leq\alpha^2\|\boldsymbol{x}-\boldsymbol{x}'\|^2.
\end{align}
Thus,
\begin{equation}
\label{e:8942a}
\text{$\boldsymbol{E}_n$ is $\eta$-strongly monotone and
$\alpha$-Lipschitzian}.
\end{equation}
Similarly,
\begin{equation}
\label{e:8942b}
\begin{cases}
\text{$\boldsymbol{F}_{\!n}$ is $\eta$-strongly monotone and
$\chi$-Lipschitzian}\\
\text{$\boldsymbol{G}_n$ is $\eta$-strongly monotone and
$\beta$-Lipschitzian}.
\end{cases}
\end{equation}
In turn, invoking \eqref{e:1372a},
\ref{t:1a}, \ref{t:1b}, and \eqref{e:1688}, we deduce that
$\boldsymbol{\mathsf{T}}_n$ is strongly monotone
with constant $\varsigma$ and Lipschitzian with constant
$\max\{\alpha,\beta,\chi,1/\varepsilon\}$.
It therefore follows from \eqref{e:1372a} and \eqref{e:1688} that
\begin{equation}
\label{e:5211}
\text{$\boldsymbol{\mathsf{K}}_n$ is $\varsigma$-strongly monotone
and $\varpi$-Lipschitzian}.
\end{equation}
Let us define
\begin{equation}
\label{e:1372c}
\begin{cases}
(\forall i\in I)\;\;E_{i,n}\colon\HH_i\to\HH_i\colon
x_i\mapsto\gamma_{i,\ell_i(n)}^{-1}x_i
-\nabla\psi_i(x_i)\\
(\forall k\in K)\;\;G_{k,n}\colon\GG_k\to\GG_k\colon
z_k\mapsto\nu_{k,\vartheta_k(n)}^{-1}z_k-\nabla h_k(z_k)
\end{cases}
\end{equation}
and let us introduce the variables
\begin{equation}
\label{e:xxyy}
\hspace{-2mm}
\begin{cases}
\boldsymbol{\mathsf{x}}_n
=(\boldsymbol{x}_n,\boldsymbol{y}_n,\boldsymbol{z}_n,
\boldsymbol{u}_n^*,\boldsymbol{v}_n^*),\;\:
\boldsymbol{\mathsf{y}}_n
=(\boldsymbol{a}_n,\boldsymbol{q}_n,\boldsymbol{b}_n,
\boldsymbol{c}_n^*,\boldsymbol{e}_n^*),\;\:
\boldsymbol{\mathsf{y}}_n^*
=(\boldsymbol{a}_n^*,\boldsymbol{q}_n^*,\boldsymbol{b}_n^*,
\boldsymbol{c}_n,\boldsymbol{e}_n)
\medskip\\
(\forall i\in I)\;\;
\begin{cases}
\widetilde{x}_{i,n}^*
=E_{i,n}x_{i,\ell_i(n)}-E_{i,n}x_n
+M_i^*\big(u_{i,n}^*-u_{i,\ell_i(n)}^*\big)
+\sum_{k\in K}L_{k,i}^*\big(v_{k,n}^*-v_{k,\ell_i(n)}^*\big)\\
\widetilde{q}_{i,n}^*
=\mu_{i,\ell_i(n)}^{-1}\big(y_{i,\ell_i(n)}
-y_{i,n}\big)+\pnabla{i}\boldsymbol{f}_{\!i}(\boldsymbol{y}_n)
-\pnabla{i}\boldsymbol{f}_{\!i}\big(\boldsymbol{y}_{\ell_i(n)}\big)
+u_{i,\ell_i(n)}^*-u_{i,n}^*\\
\widetilde{c}_{i,n}^*
=\sigma_{i,\ell_i(n)}^{-1}\big(u_{i,\ell_i(n)}^*-u_{i,n}^*\big)
+M_i\big(x_{i,\ell_i(n)}-x_{i,n}\big)-y_{i,\ell_i(n)}+y_{i,n}
\end{cases}
\medskip\\
(\forall k\in K)\;\;
\begin{cases}
\widetilde{d}_{k,n}^*
=G_{k,n}z_{k,\vartheta_k(n)}-G_{k,n}z_{k,n}
+v_{k,\vartheta_k(n)}^*-v_{k,n}^*\\
\widetilde{e}_{k,n}^*
=\varrho_{k,\vartheta_k(n)}^{-1}
\big(v_{k,\vartheta_k(n)}^*-v_{k,n}^*\big)
-z_{k,\vartheta_k(n)}+z_{k,n}
+\sum_{i\in I}L_{k,i}\big(x_{i,\vartheta_k(n)}-x_{i,n}\big)
\end{cases}
\medskip\\
\boldsymbol{\mathsf{e}}_n^*=(
\widetilde{\boldsymbol{x}}_n^*,
\widetilde{\boldsymbol{q}}_n^*,
\widetilde{\boldsymbol{d}}_n^*,
\widetilde{\boldsymbol{c}}_n^*,
\widetilde{\boldsymbol{e}}_n^*).
\end{cases}
\end{equation}
Note that, by \eqref{e:dai}, \eqref{e:4469a}, and \eqref{e:4469b},
we have
\begin{equation}
\label{e:1890}
\begin{cases}
(\forall i\in I)\;\;
q_{i,n}=q_{i,\overline{\ell}_i(n)},\;\:
c_{i,n}^*=c_{i,\overline{\ell}_i(n)}^*,\;\:
a_{i,n}=a_{i,\overline{\ell}_i(n)},\;\:
s_{i,n}^*=s_{i,\overline{\ell}_i(n)}^*,\;\:
c_{i,n}=c_{i,\overline{\ell}_i(n)}\\
(\forall k\in K)\;\;
b_{k,n}=b_{k,\overline{\vartheta}_k(n)},\;\:
e_{k,n}^*=e_{k,\overline{\vartheta}_k(n)}^*,\;\:
b_{k,n}^*=b_{k,\overline{\vartheta}_k(n)}^*.
\end{cases}
\end{equation}
Hence, for every $i\in I$, we deduce from 
\eqref{e:dai}, \eqref{e:4469a},
and \eqref{e:1372c} that
\begin{align}
\label{e:6986}
\gamma_{i,\ell_i(n)}^{-1}x_{i,\overline{\ell}_i(n)}^*
&=\gamma_{i,\ell_i(n)}^{-1}x_{i,\ell_i(n)}
-\nabla\psi_i\big(x_{i,\ell_i(n)}\big)-M_i^*u_{i,\ell_i(n)}^*
-\sum_{k\in K}L_{k,i}^*v_{k,\ell_i(n)}^*
\nonumber\\
&=E_{i,n}x_{i,\ell_i(n)}-M_i^*u_{i,\ell_i(n)}^*
-\sum_{k\in K}L_{k,i}^*v_{k,\ell_i(n)}^*
\nonumber\\
&=E_{i,n}x_{i,n}-M_i^*u_{i,n}^*-\sum_{k\in K}L_{k,i}^*v_{k,n}^*
+\widetilde{x}_{i,n}^*,
\end{align}
that
\begin{align}
\label{e:6987}
\mu_{i,\ell_i(n)}^{-1}q_{i,n}
&=\mu_{i,\ell_i(n)}^{-1}q_{i,\overline{\ell}_i(n)}
\nonumber\\
&=\mu_{i,\ell_i(n)}^{-1}y_{i,\ell_i(n)}
-\pnabla{i}\boldsymbol{f}_{\!i}\big(
\boldsymbol{y}_{\ell_i(n)}\big)
+u_{i,\ell_i(n)}^*
\nonumber\\
&=\mu_{i,\ell_i(n)}^{-1}y_{i,n}
-\pnabla{i}\boldsymbol{f}_{\!i}(\boldsymbol{y}_n)
+u_{i,n}^*+\widetilde{q}_{i,n}^*,
\end{align}
and that
\begin{align}
\label{e:6988}
\sigma_{i,\ell_i(n)}^{-1}c_{i,n}^*
&=\sigma_{i,\ell_i(n)}^{-1}c_{i,\overline{\ell}_i(n)}^*
\nonumber\\
&=\sigma_{i,\ell_i(n)}^{-1}u_{i,\ell_i(n)}^*
-y_{i,\ell_i(n)}+M_ix_{i,\ell_i(n)}
\nonumber\\
&=\sigma_{i,\ell_i(n)}^{-1}u_{i,n}^*-y_{i,n}+M_ix_{i,n}
+\widetilde{c}_{i,n}^*.
\end{align}
In a similar fashion,
\begin{equation}
\label{e:6989}
(\forall k\in K)\quad
\begin{cases}
\nu_{k,\vartheta_k(n)}^{-1}d_{k,\overline{\vartheta}_k(n)}^*
=G_{k,n}z_{k,n}+v_{k,n}^*+\widetilde{d}_{k,n}^*\\
\varrho_{k,\vartheta_k(n)}^{-1}e_{k,n}^*
=\varrho_{k,\vartheta_k(n)}^{-1}v_{k,n}^*
-z_{k,n}+\sum_{i\in I}L_{k,i}x_{i,n}
+\widetilde{e}_{k,n}^*.
\end{cases}
\end{equation}
Therefore, it results from \eqref{e:1372a}, \eqref{e:1372c},
\eqref{e:xxyy}, \eqref{e:7335}, \eqref{e:8792}, \eqref{e:3851},
and Problem~\ref{prob:1}\ref{prob:1c} that
\begin{align}
\label{e:6705}
&\boldsymbol{\mathsf{K}}_n\boldsymbol{\mathsf{x}}_n
+\boldsymbol{\mathsf{e}}_n^*
\nonumber\\
&\hspace{2mm}
=\boldsymbol{\mathsf{T}}_n\boldsymbol{\mathsf{x}}_n
-\boldsymbol{\mathsf{S}}\boldsymbol{\mathsf{x}}_n
+\boldsymbol{\mathsf{e}}_n^*
\nonumber\\
&\hspace{2mm}
=\Big(
\big(\gamma_{i,\ell_i(n)}^{-1}x_{i,\overline{\ell}_i(n)}^*
\big)_{i\in I},
\big(\mu_{i,\ell_i(n)}^{-1}q_{i,n}\big)_{i\in I},
\big(\nu_{k,\vartheta_k(n)}^{-1}d_{k,\overline{\vartheta}_k(n)}^*
\big)_{k\in K},
\big(\sigma_{i,\ell_i(n)}^{-1}c_{i,n}^*\big)_{i\in I},
\big(\varrho_{k,\vartheta_k(n)}^{-1}e_{k,n}^*\big)_{k\in K}\Big).
\end{align}
On the other hand, in the light of \eqref{e:1372a},
\eqref{e:6873}, \eqref{e:6615}, \eqref{e:3851},
and \cite[Proposition~16.44]{Livre1} we get
\begin{multline}
\label{e:4211}
(\boldsymbol{\mathsf{K}}_n+\lag)^{-1}\colon\XXX\to\XXX\colon
(\boldsymbol{x}^*,\boldsymbol{y}^*,\boldsymbol{z}^*,
\boldsymbol{u},\boldsymbol{v})\mapsto
\Big(\big(\prox_{\gamma_{i,\ell_i(n)}\varphi_i}
(\gamma_{i,\ell_i(n)}x_i^*)\big)_{i\in I},
\big(\mu_{i,\ell_i(n)}y_i^*\big)_{i\in I},\\
\big(\prox_{\nu_{k,\vartheta_k(n)}g_k}
(\nu_{k,\vartheta_k(n)}z_k^*)\big)_{k\in K},
\big(\sigma_{i,\ell_i(n)}u_i\big)_{i\in I},
\big(\varrho_{k,\vartheta_k(n)}v_k\big)_{k\in K}\Big).
\end{multline}
Hence, since \eqref{e:1890}, \eqref{e:dai}, \eqref{e:4469a},
and \eqref{e:4469b} entail that
\begin{equation}
\label{e:6990}
\begin{cases}
(\forall i\in I)\;\;a_{i,n}
=a_{i,\overline{\ell}_i(n)}
=\prox_{\gamma_{i,\ell_i(n)}\varphi_i}x_{i,\overline{\ell}_i(n)}^*
\\
(\forall k\in K)\;\;b_{k,n}
=b_{k,\overline{\vartheta}_k(n)}
=\prox_{\nu_{k,\vartheta_k(n)}g_k}
d_{k,\overline{\vartheta}_k(n)}^*,
\end{cases}
\end{equation}
we invoke \eqref{e:xxyy} to get
\begin{equation}
\label{e:7560}
\boldsymbol{\mathsf{y}}_n
=(\boldsymbol{\mathsf{K}}_n+\lag)^{-1}(
\boldsymbol{\mathsf{K}}_n\boldsymbol{\mathsf{x}}_n
+\boldsymbol{\mathsf{e}}_n^*).
\end{equation}
At the same time, it follows from \eqref{e:5211} and
\cite[Corollary~20.28 and Proposition~22.11(ii)]{Livre1}
that $\boldsymbol{\mathsf{K}}_n$ is surjective
and, in turn, that there exists
$\widetilde{\boldsymbol{\mathsf{x}}}_n\in\XXX$ such that
\begin{equation}
\label{e:6088}
\boldsymbol{\mathsf{K}}_n\widetilde{\boldsymbol{\mathsf{x}}}_n
=\boldsymbol{\mathsf{K}}_n\boldsymbol{\mathsf{x}}_n
+\boldsymbol{\mathsf{e}}_n^*.
\end{equation}
Thus, \eqref{e:7560} and Definition~\ref{d:warped} yield
\begin{equation}
\label{e:7561}
\boldsymbol{\mathsf{y}}_n
=J_{\lag}^{\boldsymbol{\mathsf{K}}_n}
\widetilde{\boldsymbol{\mathsf{x}}}_n.
\end{equation}
In view of \eqref{e:dai}, \eqref{e:1890}, and \eqref{e:4469a},
we derive from \eqref{e:6986} that
\begin{align}
\label{e:1174}
(\forall i\in I)\quad
a_{i,n}^*
&=s_{i,n}^*+\sum_{k\in K}L_{k,i}^*e_{k,n}^*
\nonumber\\
&=s_{i,\overline{\ell}_i(n)}^*+\sum_{k\in K}L_{k,i}^*e_{k,n}^*
\nonumber\\
&=\gamma_{i,\ell_i(n)}^{-1}\big(
x_{i,\overline{\ell}_i(n)}^*-a_{i,\overline{\ell}_i(n)}\big)
+\nabla\psi_i\big(a_{i,\overline{\ell}_i(n)}\big)
+M_i^*c_{i,\overline{\ell}_i(n)}^*+\sum_{k\in K}L_{k,i}^*e_{k,n}^*
\nonumber\\
&=\gamma_{i,\ell_i(n)}^{-1}x_{i,\overline{\ell}_i(n)}^*
-\Bigg(\gamma_{i,\ell_i(n)}^{-1}a_{i,n}-\nabla\psi_i(a_{i,n})
-M_i^*c_{i,n}^*-\sum_{k\in K}L_{k,i}^*e_{k,n}^*\Bigg)
\nonumber\\
&=\Bigg(E_{i,n}x_{i,n}-M_i^*u_{i,n}^*
-\sum_{k\in K}L_{k,i}^*v_{k,n}^*\Bigg)
\nonumber\\
&\quad\;
-\Bigg(E_{i,n}a_{i,n}-M_i^*c_{i,n}^*
-\sum_{k\in K}L_{k,i}^*e_{k,n}^*\Bigg)+\widetilde{x}_{i,n}^*,
\end{align}
from \eqref{e:6987} that
\begin{align}
\label{e:1175}
(\forall i\in I)\quad
q_{i,n}^*
&=\pnabla{i}\boldsymbol{f}_{\!i}(\boldsymbol{q}_n)-c_{i,n}^*
\nonumber\\
&=\big(\mu_{i,\ell_i(n)}^{-1}y_{i,n}
-\pnabla{i}\boldsymbol{f}_{\!i}\big(\boldsymbol{y}_n\big)
+u_{i,n}^*\big)
-\big(\mu_{i,\ell_i(n)}^{-1}q_{i,n}
-\pnabla{i}\boldsymbol{f}_{\!i}(\boldsymbol{q}_n)
+c_{i,n}^*\big)
+\widetilde{q}_{i,n}^*,
\end{align}
and from \eqref{e:6988} that
\begin{align}
\label{e:1176}
(\forall i\in I)\quad
c_{i,n}
&=c_{i,\overline{\ell}_i(n)}
\nonumber\\
&=q_{i,\overline{\ell}_i(n)}-M_ia_{i,\overline{\ell}_i(n)}
\nonumber\\
&=q_{i,n}-M_ia_{i,n}
\nonumber\\
&=\big(\sigma_{i,\ell_i(n)}^{-1}u_{i,n}^*-y_{i,n}+M_ix_{i,n}\big)
-\big(\sigma_{i,\ell_i(n)}^{-1}c_{i,n}^*-q_{i,n}
+M_ia_{i,n}\big)+\widetilde{c}_{i,n}^*.
\end{align}
A similar analysis shows that
\begin{equation}
\label{e:1177}
(\forall k\in K)\quad
b_{k,n}^*=\big(G_{k,n}z_{k,n}+v_{k,n}^*\big)-
\big(G_{k,n}b_{k,n}+e_{k,n}^*\big)+\widetilde{d}_{k,n}^*
\end{equation}
and
\begin{equation}
\label{e:1178}
(\forall k\in K)\quad
e_{k,n}=\Bigg(\varrho_{k,\vartheta_k(n)}^{-1}v_{k,n}^*
-z_{k,n}+\sum_{i\in I}L_{k,i}x_{i,n}\Bigg)
-\Bigg(\varrho_{k,\vartheta_k(n)}^{-1}e_{k,n}^*
-b_{k,n}+\sum_{i\in I}L_{k,i}a_{i,n}\Bigg)+\widetilde{e}_{k,n}^*.
\end{equation}
Altogether, it follows from \eqref{e:xxyy},
\eqref{e:1174}--\eqref{e:1178}, \eqref{e:1372c},
\eqref{e:1372a}, \eqref{e:7335}, \eqref{e:8792},
\eqref{e:3851}, and \eqref{e:6088} that
\begin{equation}
\label{e:7561h}
\boldsymbol{\mathsf{y}}_n^*
=(\boldsymbol{\mathsf{T}}_n\boldsymbol{\mathsf{x}}_n
-\boldsymbol{\mathsf{S}}\boldsymbol{\mathsf{x}}_n)
-(\boldsymbol{\mathsf{T}}_n\boldsymbol{\mathsf{y}}_n
-\boldsymbol{\mathsf{S}}\boldsymbol{\mathsf{y}}_n)
+\boldsymbol{\mathsf{e}}_n^*
=\boldsymbol{\mathsf{K}}_n\boldsymbol{\mathsf{x}}_n
-\boldsymbol{\mathsf{K}}_n\boldsymbol{\mathsf{y}}_n
+\boldsymbol{\mathsf{e}}_n^*
=\boldsymbol{\mathsf{K}}_n\widetilde{\boldsymbol{\mathsf{x}}}_n
-\boldsymbol{\mathsf{K}}_n\boldsymbol{\mathsf{y}}_n.
\end{equation}
Further, in view of \eqref{e:dai} and \eqref{e:xxyy}, we have
\begin{equation}
\label{e:1171}
\pi_n=\scal{\boldsymbol{\mathsf{y}}_n
-\boldsymbol{\mathsf{x}}_n}{\boldsymbol{\mathsf{y}}_n^*}
\quad\text{and}\quad
\boldsymbol{\mathsf{x}}_{n+1}=
\begin{cases}
\displaystyle
\boldsymbol{\mathsf{x}}_n+\frac{\lambda_n\pi_n}{
\|\boldsymbol{\mathsf{y}}_n^*\|^2}\,\boldsymbol{\mathsf{y}}_n^*,
&\text{if}\;\:\pi_n<0;\\
\boldsymbol{\mathsf{x}}_n,
&\text{otherwise}.
\end{cases}
\end{equation}
Combining \eqref{e:6873}, \eqref{e:zerL}, \eqref{e:5211},
\eqref{e:7561}, \eqref{e:7561h}, and \eqref{e:1171},
we conclude that \eqref{e:dai} is an instantiation of
\eqref{e:8824}. Hence, Proposition~\ref{p:5}\ref{p:5i} yields
\begin{equation}
\label{e:5388}
\sum_{n\in\NN}\|\boldsymbol{\mathsf{x}}_{n+1}
-\boldsymbol{\mathsf{x}}_n\|^2<\pinf.
\end{equation}
For every $i\in I$ and every integer $n\geq P$,
\eqref{e:1093} entails that $i\in\bigcup_{j=n-P}^nI_j$
and, in turn, \eqref{e:4469a} and \ref{t:1a} imply that
$n-P-D\leq\overline{\ell}_i(n)-D
\leq\tau_i(\overline{\ell}_i(n))=\ell_i(n)
\leq\overline{\ell}_i(n)\leq n$. Consequently,
\begin{align}
(\forall i\in I)(\forall n\in\NN)\quad
n\geq P+D\quad\Rightarrow\quad
\|\boldsymbol{\mathsf{x}}_n-\boldsymbol{\mathsf{x}}_{\ell_i(n)}\|
\leq\sum_{j=0}^{P+D}
\|\boldsymbol{\mathsf{x}}_n-\boldsymbol{\mathsf{x}}_{n-j}\|,
\end{align}
and we therefore infer from \eqref{e:5388} that
\begin{equation}
\label{e:4640a}
(\forall i\in I)\quad
\boldsymbol{\mathsf{x}}_n-\boldsymbol{\mathsf{x}}_{\ell_i(n)}
\to\boldsymbol{\mathsf{0}}.
\end{equation}
Likewise,
\begin{equation}
\label{e:4640b}
(\forall k\in K)\quad
\boldsymbol{\mathsf{x}}_n-\boldsymbol{\mathsf{x}}_{\vartheta_k(n)}
\to\boldsymbol{\mathsf{0}}.
\end{equation}
Hence, we deduce from \eqref{e:xxyy}, \eqref{e:1372c},
\eqref{e:1372a}, and \eqref{e:8942a} that
\begin{align}
\label{e:1503}
(\forall i\in I)\quad
\|\widetilde{x}_{i,n}^*\|
&\leq\|E_{i,n}x_{i,\ell_i(n)}-E_{i,n}x_n\|
+\|M_i^*\|\,\|u_{i,n}^*-u_{i,\ell_i(n)}^*\|
\nonumber\\
&\quad\;+
\sum_{k\in K}\|L_{k,i}^*\|\,\|v_{k,n}^*-v_{k,\ell_i(n)}^*\|
\nonumber\\
&\leq\|\boldsymbol{E}_n\boldsymbol{x}_{\ell_i(n)}
-\boldsymbol{E}_n\boldsymbol{x}_n\|
+\|M_i^*\|\,\|\boldsymbol{u}_n^*-\boldsymbol{u}_{\ell_i(n)}^*\|
+\sum_{k\in K}\|L_{k,i}^*\|\,
\|\boldsymbol{v}_n^*-\boldsymbol{v}_{\ell_i(n)}^*\|
\nonumber\\
&\leq\alpha\|\boldsymbol{x}_{\ell_i(n)}-\boldsymbol{x}_n\|
+\|M_i^*\|\,\|\boldsymbol{u}_n^*-\boldsymbol{u}_{\ell_i(n)}^*\|
+\sum_{k\in K}\|L_{k,i}^*\|\,
\|\boldsymbol{v}_n^*-\boldsymbol{v}_{\ell_i(n)}^*\|
\nonumber\\
&\to 0.
\end{align}
Moreover, using \eqref{e:xxyy}, \ref{t:1a}, and \eqref{e:4640a}, we
get
\begin{align}
\label{e:1504}
(\forall i\in I)\quad
\|\widetilde{q}_{i,n}^*\|
&\leq\mu_{i,\ell_i(n)}^{-1}\|y_{i,\ell_i(n)}-y_{i,n}\|
+\big\|\pnabla{i}\boldsymbol{f}_{\!i}(\boldsymbol{y}_n)
-\pnabla{i}\boldsymbol{f}_{\!i}
\big(\boldsymbol{y}_{\ell_i(n)}\big)\big\|
+\|u_{i,\ell_i(n)}^*-u_{i,n}^*\|
\nonumber\\
&\leq\varepsilon^{-1}\|\boldsymbol{y}_{\ell_i(n)}
-\boldsymbol{y}_n\|
+\|\boldsymbol{Q}\boldsymbol{y}_n
-\boldsymbol{Q}\boldsymbol{y}_{\ell_i(n)}\|
+\|\boldsymbol{u}_{\ell_i(n)}^*-\boldsymbol{u}_n^*\|
\nonumber\\
&\leq(\varepsilon^{-1}+\kappa)\|\boldsymbol{y}_{\ell_i(n)}
-\boldsymbol{y}_n\|
+\|\boldsymbol{u}_{\ell_i(n)}^*-\boldsymbol{u}_n^*\|
\nonumber\\
&\to 0
\end{align}
and
\begin{align}
\label{e:1505}
(\forall i\in I)\quad\|\widetilde{c}_{i,n}^*\|
&\leq
\sigma_{i,\ell_i(n)}^{-1}\|u_{i,\ell_i(n)}^*-u_{i,n}^*\|
+\|M_i\|\,\|x_{i,\ell_i(n)}-x_{i,n}\|
+\|y_{i,\ell_i(n)}-y_{i,n}\|
\nonumber\\
&\leq\varepsilon^{-1}\|\boldsymbol{u}_{\ell_i(n)}^*
-\boldsymbol{u}_n^*\|
+\|M_i\|\,\|\boldsymbol{x}_{\ell_i(n)}-\boldsymbol{x}_n\|
+\|\boldsymbol{y}_{\ell_i(n)}-\boldsymbol{y}_n\|
\nonumber\\
&\to 0.
\end{align}
A similar analysis shows that
\begin{equation}
\label{e:1506}
(\forall k\in K)\quad
\|\widetilde{d}_{k,n}^*\|\to 0\quad\text{and}\quad
\|\widetilde{e}_{k,n}^*\|\to 0.
\end{equation}
Altogether, we invoke \eqref{e:xxyy} and
\eqref{e:1503}--\eqref{e:1506} to get
\begin{equation}
\label{e:e}
\boldsymbol{\mathsf{e}}_n^*\to\boldsymbol{\mathsf{0}}.
\end{equation}
Hence, arguing as in \eqref{e:1878}, \eqref{e:5211} 
and \eqref{e:6088} give
\begin{equation}
\label{e;5t}
\|\widetilde{\boldsymbol{\mathsf{x}}}_n
-\boldsymbol{\mathsf{x}}_n\|
\leq\frac{
\|\boldsymbol{\mathsf{K}}_n\widetilde{\boldsymbol{\mathsf{x}}}_n
-\boldsymbol{\mathsf{K}}_n\boldsymbol{\mathsf{x}}_n\|}{\varsigma}
=\frac{\|\boldsymbol{\mathsf{e}}_n^*\|}{\varsigma}\to 0.
\end{equation}
Hence, Proposition~\ref{p:5}\ref{p:5ii} asserts that there exists
$\overline{\boldsymbol{\mathsf{x}}}=
(\overline{\boldsymbol{x}},\overline{\boldsymbol{y}},
\overline{\boldsymbol{z}},\overline{\boldsymbol{u}}^*,
\overline{\boldsymbol{v}}^*)\in\zer\lag$ such that
$\boldsymbol{\mathsf{x}}_n\weakly
\overline{\boldsymbol{\mathsf{x}}}$. This yields
$\boldsymbol{x}_n\weakly\overline{\boldsymbol{x}}$.
It remains to verify that $\overline{\boldsymbol{x}}$ solves
\eqref{e:nash}. Towards this end, let $i\in I$ and set
\begin{equation}
\label{e:8924}
f_i=\boldsymbol{f}_{\!i}\big(\mute;
(\boldsymbol{M}\overline{\boldsymbol{x}})_{\smallsetminus i}\big)
\quad\text{and}\quad
(\forall k\in K)\;\;
\widetilde{g}_k=(g_k+h_k)(\mute+\boldsymbol{L}_{k,\smallsetminus i}
\overline{\boldsymbol{x}}_{\smallsetminus i}).
\end{equation}
Then, by Problem~\ref{prob:1}\ref{prob:1d},
$f_i\colon\KK_i\to\RR$ is convex and G\^ateaux differentiable,
with $\nabla f_i(M_i\overline{x}_i)=
\pnabla{i}\boldsymbol{f}_{\!i}
(\boldsymbol{M}\overline{\boldsymbol{x}})$.
In addition, $(\forall k\in K)(\forall z_k\in\GG_k)$
$\partial\widetilde{g}_k(z_k)
=(\partial g_k+\nabla h_k)(z_k+\boldsymbol{L}_{k,\smallsetminus i}
\overline{\boldsymbol{x}}_{\smallsetminus i})$.
At the same time, we deduce from \eqref{e:lag} that
$\overline{\boldsymbol{u}}^*
=\boldsymbol{Q}\overline{\boldsymbol{y}}
=\boldsymbol{Q}(\boldsymbol{M}\overline{\boldsymbol{x}})$,
$\overline{\boldsymbol{z}}
=\boldsymbol{L}\overline{\boldsymbol{x}}$,
$\overline{\boldsymbol{v}}^*
\in\boldsymbol{B}\overline{\boldsymbol{z}}$, and
$\boldsymbol{0}\in\boldsymbol{A}\overline{\boldsymbol{x}}
+\boldsymbol{M}^*\overline{\boldsymbol{u}}^*
+\boldsymbol{L}^*\overline{\boldsymbol{v}}^*$.
Thus, it results from \eqref{e:3851} and
Problem~\ref{prob:1}\ref{prob:1c} that
\begin{equation}
\begin{cases}
\overline{u}_i^*=\pnabla{i}\boldsymbol{f}_{\!i}
(\boldsymbol{M}\overline{\boldsymbol{x}})
=\nabla f_i(M_i\overline{x}_i)\\
(\forall k\in K)\;\;\overline{z}_k
=\sum_{j\in I}L_{k,j}\overline{x}_j
=L_{k,i}\overline{x}_i
+\boldsymbol{L}_{k,\smallsetminus i}
\overline{\boldsymbol{x}}_{\smallsetminus i}\\
(\forall k\in K)\;\;\overline{v}_k^*\in
\partial g_k(\overline{z}_k)+\nabla h_k(\overline{z}_k)
=(\partial g_k+\nabla h_k)(L_{k,i}\overline{x}_i
+\boldsymbol{L}_{k,\smallsetminus i}
\overline{\boldsymbol{x}}_{\smallsetminus i})
=\partial\widetilde{g}_k(L_{k,i}\overline{x}_i)
\end{cases}
\end{equation}
and, in turn, from \eqref{e:8792} and
\cite[Proposition~16.6(ii)]{Livre1} that
\begin{align}
0&\in\partial\varphi_i(\overline{x}_i)+\nabla\psi_i(\overline{x}_i)
+M_i^*\overline{u}_i^*+\sum_{k\in K}L_{k,i}^*\overline{v}_k^*
\nonumber\\
&\subset\partial\varphi_i(\overline{x}_i)
+\nabla\psi_i(\overline{x}_i)
+M_i^*\big(\nabla f_i(M_i\overline{x}_i)\big)
+\sum_{k\in K}L_{k,i}^*\big(
\partial\widetilde{g}_k(L_{k,i}\overline{x}_i)\big)
\nonumber\\
&\subset\partial\Bigg(\varphi_i+\psi_i
+f_i\circ M_i+\sum_{k\in K}\widetilde{g}_k\circ L_{k,i}\Bigg)
(\overline{x}_i).
\end{align}
Consequently, appealing to Fermat's rule
\cite[Theorem~16.3]{Livre1} and \eqref{e:8924}, we arrive at
\begin{align}
\overline{x}_i
&\in\Argmind{x_i\in\HH_i}{
\varphi_i(x_i)+\psi_i(x_i)+f_i(M_ix_i)
+\sum_{k\in K}\widetilde{g}_k(L_{k,i}x_i)}
\nonumber\\
&=\Argmind{x_i\in\HH_i}{
\varphi_i(x_i)+\psi_i(x_i)+\boldsymbol{f}_{\!i}\big(M_ix_i;
(\boldsymbol{M}\overline{\boldsymbol{x}})_{\smallsetminus i}\big)
+\sum_{k\in K}(g_k+h_k)(L_{k,i}x_i
+\boldsymbol{L}_{k,\smallsetminus i}
\overline{\boldsymbol{x}}_{\smallsetminus i})},
\end{align}
which completes the proof.
\end{proof}

\begin{remark}
\label{r:8}
Let us confirm that algorithm \eqref{e:dai} complies with the
principles laid out in Section~\ref{sec:1}.
\begin{itemize}
\item
{\bfseries Decomposition:}
In \eqref{e:dai}, the nonsmooth functions
$(\varphi_i)_{i\in I}$ and $(g_k)_{k\in K}$
are activated separately via their proximity operators,
while the smooth functions
$(\psi_i)_{i\in I}$, $(\boldsymbol{f}_{\!i})_{i\in I}$,
and $(h_k)_{k\in K}$ are activated separately via their gradients.
\item
{\bfseries Block-iterative implementation:}
At any iteration $n$, the functions 
$(\boldsymbol{f}_{\!i})_{i\in I}$ are activated and we require only
that the subfamilies $(\varphi_i)_{i\in I_n}$, 
$(\psi_i)_{i\in I_n}$, $(g_k)_{k\in K_n}$, and $(h_k)_{k\in K_n}$
be used. To guarantee convergence, we ask in 
condition~\ref{t:1c} of Theorem~\ref{t:1} that each of these
functions be activated frequently enough.
\item
{\bfseries Asynchronous implementation:}
Given $i\in I$ and $k\in K$, the asynchronous character of the
algorithm is materialized by the variables $\tau_i(n)$ and
$\delta_k(n)$ which signal when the underlying computations
incorporated at iteration $n$ were initiated. 
Conditions~\ref{t:1a} and \ref{t:1b} of Theorem~\ref{t:1} ask that
the lag between the initiation and the incorporation of such
computations do not exceed $D$ iterations. The introduction of such
techniques in monotone operator splitting were initiated in
\cite{MaPr18}.
\end{itemize}
\end{remark}

\begin{remark}
\label{r:3}
Consider the proof of Theorem~\ref{t:1}. Since 
Proposition~\ref{p:5}\ref{p:5ii} yields
$\boldsymbol{\mathsf{x}}_n-\boldsymbol{\mathsf{y}}_n\to
\boldsymbol{\mathsf{0}}$, we obtain 
$\boldsymbol{x}_n-\boldsymbol{a}_n\to\boldsymbol{0}$
via \eqref{e:xxyy} and thus 
$\boldsymbol{a}_n\weakly\overline{\boldsymbol{x}}$.
At the same time, by \eqref{e:dai}, given $i\in I$, the sequence
$(a_{i,n})_{n\in\NN}$ lies in 
$\dom\partial\varphi_i\subset\dom\varphi_i$.
In particular, if a constraint on $\overline{x}_i$ is enforced via
$\varphi_i=\iota_{C_i}$, then $(a_{i,n})_{n\in\NN}$ converges 
to the $i$th component of a solution $\overline{\boldsymbol{x}}$
while being feasible in the sense that 
$C_i\ni a_{i,n}\weakly\overline{x}_i$.
\end{remark}

\begin{remark}
The proof of Theorem~\ref{t:1} implicitly establishes the
convergence of an asynchronous block-iterative algorithm to solve
the more general system of monotone inclusions
\begin{multline}
\label{e:2p}
\text{find}\;\:\overline{\boldsymbol{x}}\in\HHH\;\:
\text{such that}\\
(\forall i\in I)\;\;0\in A_i\overline{x}_i+R_i\overline{x}_i
+M_i^*\big(Q_i(\boldsymbol{M}\overline{\boldsymbol{x}})\big)
+\Sum_{k\in K}L_{k,i}^*\Bigg((B_k+D_k)
\Bigg(\Sum_{j\in I}L_{k,j}\overline{x}_j\Bigg)\Bigg)
\end{multline}
under the following assumptions:
\begin{enumerate}[label={\rm[\alph*]}]
\item
For every $i\in I$, $A_i\colon\HH_i\to 2^{\HH_i}$
is maximally monotone, $\alpha_i\in\RP$,
and $R_i\colon\HH_i\to\HH_i$ is monotone and
$\alpha_i$-Lipschitzian.
\item
For every $i\in I$, $Q_i\colon\KKK\to\KK_i$. It is assumed that
the operator $\boldsymbol{Q}\colon\KKK\to\KKK\colon
\boldsymbol{y}\mapsto(Q_i\boldsymbol{y})_{i\in I}$
is monotone and Lipschitzian. Furthermore, $(\chi_i)_{i\in I}$ 
are positive numbers such that
\begin{equation}
(\forall\boldsymbol{y}\in\KKK)(\forall\boldsymbol{y}'\in\KKK)
\quad\scal{\boldsymbol{y}-\boldsymbol{y}'}
{\boldsymbol{Q}\boldsymbol{y}-\boldsymbol{Q}\boldsymbol{y}'}
\leq\sum_{i\in I}\chi_i\|y_i-y_i'\|^2.
\end{equation}
\item
For every $k\in K$, $B_k\colon\GG_k\to 2^{\GG_k}$ is maximally
monotone, $\beta_k\in\RP$, and $D_k\colon\GG_k\to\GG_k$ is 
monotone and $\beta_k$-Lipschitzian.
\item
For every $i\in I$ and every $k\in K$,
$M_i\colon\HH_i\to\KK_i$ and $L_{k,i}\colon\HH_i\to\GG_k$
are linear and bounded. Moreover, we set
$\boldsymbol{M}\colon\HHH\to\KKK\colon
\boldsymbol{x}\mapsto(M_ix_i)_{i\in I}$.
\end{enumerate}
Indeed, denote by $\KT$ the set of points
$(\boldsymbol{x},\boldsymbol{u}^*,\boldsymbol{v}^*)\in
\HHH\oplus\KKK\oplus\GGG$ such that
\begin{equation}
\begin{cases}
(\forall i\in I)\;\;u_i^*=Q_i(\boldsymbol{M}\boldsymbol{x})\\
(\forall k\in K)\;\;v_k^*\in
(B_k+D_k)\big(\sum_{j\in I}L_{k,j}x_j\big)\\
(\forall i\in I)\;\;{-}M_i^*u_i^*
-\sum_{k\in K}L_{k,i}^*v_k^*\in A_ix_i+R_ix_i.
\end{cases}
\end{equation}
Suppose that $\KT\neq\emp$
and execute \eqref{e:dai} with the following modifications:
\begin{itemize}
\item
For every $i\in I$ and every $n\in\NN$,
$\prox_{\gamma_{i,n}\varphi_i}$ is replaced by
$J_{\gamma_{i,n}A_i}^{\Id}$, $\nabla\psi_i$ by $R_i$, and
$\pnabla{i}\boldsymbol{f}_{\!i}$ by $Q_i$.
\item
For every $k\in K$ and every $n\in\NN$,
$\prox_{\nu_{k,n}g_k}$ is replaced by $J_{\nu_{k,n}B_k}^{\Id}$, 
and $\nabla h_k$ by $D_k$.
\end{itemize}
Then there exists
$(\overline{\boldsymbol{x}},\overline{\boldsymbol{u}}^*,
\overline{\boldsymbol{v}}^*)\in\KT$ such that
$(\boldsymbol{x}_n,\boldsymbol{u}_n^*,
\boldsymbol{v}_n^*)\weakly
(\overline{\boldsymbol{x}},\overline{\boldsymbol{u}}^*,
\overline{\boldsymbol{v}}^*)$ and $\overline{\boldsymbol{x}}$
solves \eqref{e:2p}.
\end{remark}

\begin{remark}
By invoking \cite[Theorem~4.8]{Jmaa20} and arguing as in the proof
of Proposition~\ref{p:5}, we obtain a strongly convergent
counterpart of Proposition~\ref{p:5} which, in turn, yields
a strongly convergent version of Theorem~\ref{t:1}.
\end{remark}

Theorem~\ref{t:1} requires that \eqref{e:7592} be satisfied.  With
the assistance of monotone operator theory arguments applied to a
set of primal-dual inclusions, we provide below sufficient
conditions for that. Let us start with a technical fact.

\begin{lemma}
\label{l:3m}
Let $\HHH$ and $\GGG$ be real Hilbert spaces,
let $\boldsymbol{B}\colon\GGG\to 2^{\GGG}$ be $3^*$ monotone,
and let $\boldsymbol{L}\colon\HHH\to\GGG$ be linear and bounded.
Then $\boldsymbol{L}^*\circ\boldsymbol{B}\circ\boldsymbol{L}$
is $3^*$ monotone.
\end{lemma}
\begin{proof}
Set $\boldsymbol{A}=
\boldsymbol{L}^*\circ\boldsymbol{B}\circ\boldsymbol{L}$.
First, we deduce from \cite[Proposition~20.10]{Livre1} that
$\boldsymbol{A}$ is monotone. Next, take
$\boldsymbol{x}\in\dom\boldsymbol{A}$ and $\boldsymbol{x}^*\in
\ran\boldsymbol{A}$. On the one hand,
$\boldsymbol{L}\boldsymbol{x}\in\dom\boldsymbol{B}$
and there exists $\boldsymbol{z}^*\in\ran\boldsymbol{B}$
such that $\boldsymbol{x}^*=\boldsymbol{L}^*\boldsymbol{z}^*$.
On the other hand, for every
$(\boldsymbol{y},\boldsymbol{y}^*)\in\gra\boldsymbol{A}$,
there exists $\boldsymbol{v}^*\in\GGG$ such that
$(\boldsymbol{L}\boldsymbol{y},\boldsymbol{v}^*)\in
\gra\boldsymbol{B}$ and
$\boldsymbol{y}^*=\boldsymbol{L}^*\boldsymbol{v}^*$,
from which we obtain
\begin{align}
\scal{\boldsymbol{x}-\boldsymbol{y}}{\boldsymbol{y}^*
-\boldsymbol{x}^*}
&=\scal{\boldsymbol{x}-\boldsymbol{y}}{
\boldsymbol{L}^*\boldsymbol{v}^*-
\boldsymbol{L}^*\boldsymbol{z}^*}
\nonumber\\
&=\scal{\boldsymbol{L}\boldsymbol{x}-
\boldsymbol{L}\boldsymbol{y}}{\boldsymbol{v}^*-\boldsymbol{z}^*}
\nonumber\\
&\leq\sup_{(\boldsymbol{w},\boldsymbol{w}^*)\in
\gra\boldsymbol{B}}\scal{\boldsymbol{L}\boldsymbol{x}-
\boldsymbol{w}}{\boldsymbol{w}^*-\boldsymbol{z}^*}.
\end{align}
Therefore, by $3^*$ monotonicity of $\boldsymbol{B}$,
\begin{equation}
\sup_{(\boldsymbol{y},\boldsymbol{y}^*)\in\gra\boldsymbol{A}}
\scal{\boldsymbol{x}-\boldsymbol{y}}{\boldsymbol{y}^*
-\boldsymbol{x}^*}
\leq\sup_{(\boldsymbol{w},\boldsymbol{w}^*)\in
\gra\boldsymbol{B}}\scal{\boldsymbol{L}\boldsymbol{x}-
\boldsymbol{w}}{\boldsymbol{w}^*-\boldsymbol{z}^*}
<\pinf.
\end{equation}
Consequently, $\boldsymbol{A}$ is $3^*$ monotone.
\end{proof}

\begin{proposition}
\label{p:9}
Consider the setting of Problem~\ref{prob:1} and set
\begin{equation}
\label{e:3038}
\boldsymbol{C}=\Menge{\Bigg(
\sum_{i\in I}L_{k,i}x_i-z_k\Bigg)_{k\in K}}{
(\forall i\in I)\;\;x_i\in\dom\varphi_i\;\:\text{and}\;\:
(\forall k\in K)\;\;z_k\in\dom g_k}.
\end{equation}
Suppose that $\boldsymbol{0}\in\sri\boldsymbol{C}$
and that one of the following is satisfied: 
\begin{enumerate}[label={\rm[\alph*]}]
\item 
\label{p:9a}
For every $i\in I$, one of the following holds:
\begin{enumerate}[label={\rm\arabic*/}]
\item
\label{p:9a1}
$\partial(\varphi_i+\psi_i)$ is surjective.
\item
\label{p:9a2}
$\varphi_i+\psi_i$ is supercoercive.
\item
\label{p:9a3}
$\dom\varphi_i$ is bounded.
\item
\label{p:9a4}
$\varphi_i+\psi_i$ is uniformly convex.
\end{enumerate}
\item
\label{p:9b}
$\boldsymbol{Q}$ is $3^*$ monotone and
one of the following holds:
\begin{enumerate}[label={\rm\arabic*/}]
\item
\label{p:9b1}
$\boldsymbol{M}^*\circ\boldsymbol{Q}\circ\boldsymbol{M}$
is surjective.
\item
\label{p:9b2}
$\boldsymbol{Q}$ is surjective and,
for every $i\in I$, $M_i$ is bijective.
\end{enumerate}
\end{enumerate}
Then \eqref{e:7592} holds.
\end{proposition}
\begin{proof}
Let $\boldsymbol{A}$, $\boldsymbol{B}$, and $\boldsymbol{L}$
be as in \eqref{e:3851} and define
\begin{equation}
\label{e:d4y}
\boldsymbol{T}\colon\HHH\to 2^{\HHH}\colon\boldsymbol{x}\mapsto
\boldsymbol{A}\boldsymbol{x}+\boldsymbol{L}^*\big(
\boldsymbol{B}(\boldsymbol{L}\boldsymbol{x})\big)
+\boldsymbol{M}^* \big(\boldsymbol{Q}(
\boldsymbol{M}\boldsymbol{x})\big).
\end{equation}
Suppose that $\widehat{\boldsymbol{x}}\in\zer\boldsymbol{T}$
and set $\widehat{\boldsymbol{u}}^*
=\boldsymbol{Q}(\boldsymbol{M}\widehat{\boldsymbol{x}})$.
On the one hand,
in view of Problem~\ref{prob:1}\ref{prob:1d},
$(\forall i\in I)$ $\widehat{u}_i^*
=\pnabla{i}\boldsymbol{f}_{\!i}
(\boldsymbol{M}\widehat{\boldsymbol{x}})$.
On the other hand, it results from \eqref{e:d4y} that
there exists $\widehat{\boldsymbol{v}}^*\in
\boldsymbol{B}(\boldsymbol{L}\widehat{\boldsymbol{x}})$ such that
${-}\boldsymbol{M}^*\widehat{\boldsymbol{u}}^*
-\boldsymbol{L}^*\widehat{\boldsymbol{v}}^*
\in\boldsymbol{A}\widehat{\boldsymbol{x}}$
or, equivalently, by \eqref{e:8792} and \eqref{e:3851},
$(\forall i\in I)$ ${-}M_i^*\widehat{u}_i^*
-\sum_{k\in K}L_{k,i}^*\widehat{v}_k^*\in
\partial\varphi_i(\widehat{x}_i)+\nabla\psi_i(\widehat{x}_i)$.
Further, using \eqref{e:3851}, we obtain
$(\forall k\in K)$ $\widehat{v}_k^*\in
(\partial g_k+\nabla h_k)(\sum_{j\in I}L_{k,j}\widehat{x}_j)$.
Altogether, we have shown that
$\zer\boldsymbol{T}\neq\emp$ $\Rightarrow$ \eqref{e:7592} holds.
Therefore, it suffices to show that $\zer\boldsymbol{T}\neq\emp$.
To do so, define
\begin{equation}
\begin{cases}
\boldsymbol{\varphi}\colon\HHH\to\RX\colon\boldsymbol{x}\mapsto
\sum_{i\in I}\big(\varphi_i(x_i)+\psi_i(x_i)\big)\\
\boldsymbol{g}\colon\GGG\to\RX\colon\boldsymbol{z}\mapsto
\sum_{k\in K}\big(g_k(z_k)+h_k(z_k)\big)\\
\boldsymbol{P}=\boldsymbol{A}
+\boldsymbol{L}^*\circ\boldsymbol{B}\circ\boldsymbol{L}.
\end{cases}
\end{equation}
Then, by \eqref{e:3851} and \cite[Proposition~16.9]{Livre1},
$\boldsymbol{A}=\partial\boldsymbol{\varphi}$ and
$\boldsymbol{B}=\partial\boldsymbol{g}$.
In turn, since \eqref{e:3038} and \eqref{e:3851} imply that
$\boldsymbol{0}\in\sri\boldsymbol{C}
=\sri\!(\boldsymbol{L}(\dom\boldsymbol{\varphi})
-\dom\boldsymbol{g})$, we derive from
\cite[Theorem~16.47(i)]{Livre1} that
$\boldsymbol{P}=\boldsymbol{A}
+\boldsymbol{L}^*\circ\boldsymbol{B}\circ
\boldsymbol{L}=\partial(\boldsymbol{\varphi}
+\boldsymbol{g}\circ\boldsymbol{L})$.
Therefore, in view of \cite[Theorem~20.25 and
Example~25.13]{Livre1},
\begin{equation}
\label{e:1185}
\text{$\boldsymbol{A}$, $\boldsymbol{B}$, and
$\boldsymbol{P}$ are maximally monotone and $3^*$ monotone}.
\end{equation}

\ref{p:9a}:
Fix temporarily $i\in I$. By \cite[Theorem~20.25]{Livre1},
$\partial(\varphi_i+\psi_i)$ is maximally monotone.
First, if \ref{p:9a}\ref{p:9a2} holds, then
\cite[Corollary~16.30, and Propositions~14.15 and 16.27]{Livre1}
entail that $\ran\partial(\varphi_i+\psi_i)
=\dom\partial(\varphi_i+\psi_i)^*=\HH_i$ and, hence,
\ref{p:9a}\ref{p:9a1} holds. Second,
if \ref{p:9a}\ref{p:9a3} holds, then
$\dom\partial(\varphi_i+\psi_i)\subset\dom\!(\varphi_i+\psi_i)
=\dom\varphi_i$ is bounded and, therefore, it follows from
\cite[Corollary~21.25]{Livre1} that \ref{p:9a}\ref{p:9a1} holds.
Finally, if \ref{p:9a}\ref{p:9a4} holds, then
\cite[Proposition~17.26(ii)]{Livre1} implies that
\ref{p:9a}\ref{p:9a2} holds and, in turn, that 
\ref{p:9a}\ref{p:9a1} holds.
Altogether, it is enough to assume that the operators
$(\partial(\varphi_i+\psi_i))_{i\in I}$
are surjective and to show that $\zer\boldsymbol{T}\neq\emp$.
Assume that $(\partial(\varphi_i+\psi_i))_{i\in I}$
are surjective and set
$\boldsymbol{R}={-}\boldsymbol{M}\circ\boldsymbol{P}^{-1}
\circ({-}\boldsymbol{M}^*)+\boldsymbol{Q}^{-1}$.
Then we derive from \eqref{e:3851} that
$\boldsymbol{A}$ is surjective.
On the other hand, Lemma~\ref{l:3m} asserts that
$\boldsymbol{L}^*\circ\boldsymbol{B}\circ\boldsymbol{L}$
is $3^*$ monotone. Hence, \eqref{e:1185} and
\cite[Corollary~25.27(i)]{Livre1} yields
$\dom\boldsymbol{P}^{-1}=\ran\boldsymbol{P}=\HHH$.
In turn, since $\boldsymbol{P}^{-1}$ and $\boldsymbol{Q}^{-1}$
are maximally monotone, \cite[Theorem~25.3]{Livre1} implies that
$\boldsymbol{R}$ is likewise. Furthermore, we observe that
$\dom\boldsymbol{Q}^{-1}\subset\KKK
=\dom\!({-}\boldsymbol{M}\circ\boldsymbol{P}^{-1}
\circ({-}\boldsymbol{M}^*))$ and,
by virtue of \eqref{e:1185}, \cite[Proposition~25.19(i)]{Livre1},
and Lemma~\ref{l:3m}, that
${-}\boldsymbol{M}\circ\boldsymbol{P}^{-1}
\circ({-}\boldsymbol{M}^*)$ is $3^*$ monotone.
Therefore, since
$\ran\boldsymbol{Q}^{-1}=\dom\boldsymbol{Q}=\KKK$,
\cite[Corollary~25.27(ii)]{Livre1} entails that
$\boldsymbol{R}$ is surjective and, in turn, that
$\zer\boldsymbol{R}\neq\emp$.
Consequently, \cite[Proposition~26.33(iii)]{Livre1}
asserts that $\zer\boldsymbol{T}\neq\emp$.

\ref{p:9b}\ref{p:9b1}:
Lemma~\ref{l:3m} asserts that
$\boldsymbol{M}^*\circ\boldsymbol{Q}\circ\boldsymbol{M}$
is $3^*$ monotone. At the same time,
since $\boldsymbol{Q}$ is maximally monotone and
$\dom\boldsymbol{Q}=\KKK$, it results from \eqref{e:1185} and
\cite[Theorem~25.3]{Livre1} that $\boldsymbol{T}=\boldsymbol{P}
+\boldsymbol{M}^*\circ\boldsymbol{Q}\circ\boldsymbol{M}$
is maximally monotone. Hence, since
$\boldsymbol{M}^*\circ\boldsymbol{Q}\circ\boldsymbol{M}$
is surjective, we derive from \eqref{e:1185} and
\cite[Corollary~25.27(i)]{Livre1} that $\boldsymbol{T}$
is surjective and, therefore, that
$\zer\boldsymbol{T}\neq\emp$.

\ref{p:9b}\ref{p:9b2}$\Rightarrow$\ref{p:9b}\ref{p:9b1}:
Since the assumption implies that $\boldsymbol{M}$ is bijective,
so is $\boldsymbol{M}^*$. This makes
$\boldsymbol{M}^*\circ\boldsymbol{Q}\circ\boldsymbol{M}$
surjective.
\end{proof}

\begin{remark}
Sufficient conditions for $\boldsymbol{0}\in\sri\boldsymbol{C}$
to hold in Proposition~\ref{p:9} can be found in
\cite[Proposition~5.3]{Siop13}.
\end{remark}

\section{Application examples}
\label{sec:4}

We discuss problems which are shown to be realizations
of Problem~\ref{prob:1} and which can therefore be solved by the
asynchronous block-iterative algorithm \eqref{e:dai} of
Theorem~\ref{t:1}.

\begin{example}[quadratic coupling]
\label{ex:9}
Let $\KK$ be a real Hilbert space and
let $I$ be a nonempty finite set. For every $i\in I$,
let $\HH_i$ be a real Hilbert space,
let $\varphi_i\in\Gamma_0(\HH_i)$,
let $\alpha_i\in\RP$,
let $\psi_i\colon\HH_i\to\RR$ be convex and differentiable
with an $\alpha_i$-Lipschitzian gradient,
let $M_i\colon\HH_i\to\KK$ be linear and bounded,
let $\Lambda_i$ be a nonempty finite set,
let $(\omega_{i,\ell,j})_{\ell\in\Lambda_i,j\in
I\smallsetminus\{i\}}$ be in $\RP$,
and let $(\kappa_{i,\ell})_{\ell\in\Lambda_i}$ be in $\RPP$.
Additionally, set $\HHH=\bigoplus_{i\in I}\HH_i$
and $\KKK=\bigoplus_{i\in I}\KK$. The problem is to
\begin{multline}
\label{e:nash4}
\text{find}\;\:\overline{\boldsymbol{x}}\in\HHH\;\:
\text{such that}\\(\forall i\in I)\;\;
\overline{x}_i\in\Argmind{x_i\in\HH_i}{\varphi_i(x_i)+\psi_i(x_i)
+\sum_{\ell\in\Lambda_i}\frac{\kappa_{i,\ell}}{2}\Bigg\|
M_ix_i-\sum_{j\in I\smallsetminus \{i\}}
\omega_{i,\ell,j}M_j\overline{x}_j\Bigg\|^2}.
\end{multline}
It is assumed that
\begin{equation}
\label{e:2096}
(\forall\boldsymbol{y}\in\KKK)(\forall\boldsymbol{y}'\in\KKK)\quad
\sum_{i\in I}\sum_{\ell\in\Lambda_i}\kappa_{i,\ell}
\Scal{y_i-y_i'}{y_i-y_i'-\sum_{j\in I\smallsetminus\{i\}}
\omega_{i,\ell,j}(y_j-y_j')}\geq 0.
\end{equation}
Define
\begin{equation}
\label{e:2287}
(\forall i\in I)\quad\boldsymbol{f}_{\!i}\colon\KKK\to\RR\colon
\boldsymbol{y}\mapsto
\sum_{\ell\in\Lambda_i}\frac{\kappa_{i,\ell}}{2}\Bigg\|
y_i-\sum_{j\in I\smallsetminus \{i\}}
\omega_{i,\ell,j}y_j\Bigg\|^2.
\end{equation}
Then, for every $i\in I$ and every $\boldsymbol{y}\in\KKK$,
$\boldsymbol{f}_{\!i}(\mute;\boldsymbol{y}_{\smallsetminus i})$
is convex and differentiable with
\begin{equation}
\pnabla{i}\boldsymbol{f}_{\!i}(\boldsymbol{y})
=\sum_{\ell\in\Lambda_i}\kappa_{i,\ell}\Bigg(
y_i-\sum_{j\in I\smallsetminus \{i\}}
\omega_{i,\ell,j}y_j\Bigg).
\end{equation}
Hence, in view of \eqref{e:2096}, the operator
$\boldsymbol{Q}\colon\KKK\to\KKK\colon
\boldsymbol{y}\mapsto
(\pnabla{i}\boldsymbol{f}_{\!i}(\boldsymbol{y}))_{i\in I}$
is monotone and Lipschitzian.
Thus, \eqref{e:nash4} is a special case of \eqref{e:nash}
with $(\forall i\in I)$ $\KK_i=\KK$ and
$(\forall k\in K)$ $g_k=h_k=0$.
In particular, suppose that, for every $i\in I$, $\HH_i=\KK$,
$C_i$ is a nonempty closed convex subset of $\HH_i$,
$\varphi_i=\iota_{C_i}$, $\psi_i=0$, $M_i=\Id$,
$\Lambda_i\subset I\smallsetminus\{i\}$, and
\begin{equation}
(\forall\ell\in\Lambda_i)\quad
\begin{cases}
\kappa_{i,\ell}=1\\
(\forall j\in I\smallsetminus\{i\})\;\;\omega_{i,\ell,j}=
\begin{cases}
1,&\text{if}\;\:j=\ell;\\
0,&\text{if}\;\:j\neq\ell.
\end{cases}
\end{cases}
\end{equation}
Then \eqref{e:nash4} becomes
\begin{equation}
\label{e:nash14}
\text{find}\;\:\overline{\boldsymbol{x}}\in\HHH\;\:
\text{such that}\;\:(\forall i\in I)\;\;
\overline{x}_i\in\Argmind{x_i\in C_i}{\frac{1}{2}
\sum_{\ell\in\Lambda_i}\|x_i-\overline{x}_{\ell}\|^2}.
\end{equation}
This unifies models found in \cite{Bail12}.
\end{example}

\begin{example}[minimax]
\label{ex:4}
Let $I$ be a finite set and suppose that $\emp\neq J\subset I$.
Let $(\HH_i)_{i\in I}$ be real Hilbert spaces, and set
$\UUU=\bigoplus_{i\in I\smallsetminus J}\HH_i$ and
$\VVV=\bigoplus_{j\in J}\HH_j$.
For every $i\in I$, let $\varphi_i\in\Gamma_0(\HH_i)$,
let $\alpha_i\in\RP$, let $\psi_i\colon\HH_i\to\RR$ be convex and
differentiable with an $\alpha_i$-Lipschitzian gradient.
Further, let
$\boldsymbol{\EuScript{L}}\colon\UUU\oplus\VVV\to\RR$ be
differentiable with a Lipschitzian gradient and 
such that, for every $\boldsymbol{u}\in\UUU$ and every
$\boldsymbol{v}\in\VVV$, the functions
${-}\boldsymbol{\EuScript{L}}(\boldsymbol{u},\mute)$ and 
$\boldsymbol{\EuScript{L}}(\mute,\boldsymbol{v})$ are
convex. Finally, for every $i\in I\smallsetminus J$ and every 
$j\in J$, let $L_{j,i}\colon\HH_i\to\HH_j$ be linear and bounded.
Consider the multivariate minimax problem
\begin{equation}
\label{e:mima}
\minmax{\boldsymbol{u}\in\UUU}{\boldsymbol{v}\in\VVV}{
\sum_{i\in I\smallsetminus J}\big(\varphi_i(u_i)+\psi_i(u_i)\big)
-\sum_{j\in J}\big(\varphi_j(v_j)+\psi_j(v_j)\big)
+\boldsymbol{\EuScript{L}}(\boldsymbol{u},\boldsymbol{v})
+\sum_{i\in I\smallsetminus J}\sum_{j\in J}\scal{L_{j,i}u_i}{v_j}
}.
\end{equation}
Now set $\HHH=\UUU\oplus\VVV$ and define
\begin{equation}
(\forall i\in I)\quad
\boldsymbol{f}_{\!i}\colon\HHH\to\RR\colon
(\boldsymbol{u},\boldsymbol{v})\mapsto
\begin{cases}
\boldsymbol{\EuScript{L}}(\boldsymbol{u},\boldsymbol{v})
+\sscal{u_i}{\sum_{j\in J}L_{j,i}^*v_j},
&\text{if}\;\:i\in I\smallsetminus J;\\
{-}\boldsymbol{\EuScript{L}}(\boldsymbol{u},\boldsymbol{v})
-\sscal{\sum_{k\in I\smallsetminus J}L_{i,k}u_k}{v_i},
&\text{if}\;\:i\in J.
\end{cases}
\end{equation}
Then $\HHH=\bigoplus_{i\in I}\HH_i$ and \eqref{e:mima} can be put
in the form
\begin{equation}
\text{find}\;\:\overline{\boldsymbol{x}}\in\HHH\;\:
\text{such that}\;\:(\forall i\in I)\;\;
\overline{x}_i\in\Argmind{x_i\in\HH_i}{
\varphi_i(x_i)+\psi_i(x_i)+\boldsymbol{f}_{\!i}(x_i;
\overline{\boldsymbol{x}}_{\smallsetminus i})}.
\end{equation}
Let us verify Problem~\ref{prob:1}\ref{prob:1d}.
On the one hand, we have
\begin{equation}
\label{e:78}
(\forall i\in I)(\forall\boldsymbol{x}\in\HHH)\quad
\pnabla{i}\boldsymbol{f}_{\!i}(\boldsymbol{x})=
\begin{cases}
\pnabla{i}\boldsymbol{\EuScript{L}}(\boldsymbol{x})
+\sum_{j\in J}L_{j,i}^*x_j,
&\text{if}\;\:i\in I\smallsetminus J;\\
{-}\pnabla{i}\boldsymbol{\EuScript{L}}(\boldsymbol{x})
-\sum_{k\in I\smallsetminus J}L_{i,k}x_k,
&\text{if}\;\:i\in J.
\end{cases}
\end{equation}
On the other hand, the operator
\begin{equation}
\boldsymbol{R}\colon\HHH\to\HHH\colon
\boldsymbol{x}\mapsto
\big(\big(\pnabla{i}\boldsymbol{\EuScript{L}}(\boldsymbol{x})
\big)_{i\in I\smallsetminus J},
\big({-}\pnabla{j}\boldsymbol{\EuScript{L}}
(\boldsymbol{x})\big)_{j\in J}\big)
\end{equation}
is monotone \cite{Rock70,Roc71d} and Lipschitzian,
while the bounded linear operator
\begin{equation}
\boldsymbol{S}\colon\HHH\to\HHH\colon\boldsymbol{x}\mapsto
\Bigg(
\Bigg(\Sum_{j\in J}L_{j,i}^*x_j\Bigg)_{i\in I\smallsetminus J},
\Bigg({-}\Sum_{k\in I\smallsetminus J}L_{i,k}x_k\Bigg)_{i\in J}
\Bigg)
\end{equation}
satisfies $\boldsymbol{S}^*={-}\boldsymbol{S}$
and it is therefore monotone
\cite[Example~20.35]{Livre1}.
Hence, since the operator $\boldsymbol{Q}$ in
Problem~\ref{prob:1}\ref{prob:1d} can be written as
$\boldsymbol{Q}=\boldsymbol{R}+\boldsymbol{S}$,
it is therefore monotone and Lipschitzian.
Altogether, \eqref{e:mima} is an instantiation of \eqref{e:nash}.
Special cases of \eqref{e:mima} can be found in 
\cite{Sign20,Rock95}.
\end{example}

\begin{example}
\label{ex:31}
In Problem~\ref{prob:1}, consider the following scenario: 
$K=\{1\}$, $\GG_1$ is the standard Euclidean space $\RR^M$,
$r\in\GG_1$, $g_1=\iota_{E}$, where $E=r+\RP^M$,
$h_1=0$, and, for every $i\in I$, $\HH_i$ is the standard Euclidean
space $\RR^{N_i}$, $\psi_i=0$, and 
$\varphi_i=\iota_{C_i}$, where $C_i$ is a nonempty closed convex
subset of $\HH_i$. Then, upon setting $N=\sum_{i\in I}N_i$,
we obtain the model 
\begin{equation}
\label{e:nash8}
\text{find}\;\:\overline{\boldsymbol{x}}\in\RR^N\;\:
\text{such that}\;\:(\forall i\in I)\;\;
\overline{x}_i\in\Argmind{x_i\in C_i\\
L_{1,i}x_i+\boldsymbol{L}_{1,\smallsetminus i}
\overline{\boldsymbol{x}}_{\smallsetminus i}\in E}
{\boldsymbol{f}_{\!i}(x_i;
\overline{\boldsymbol{x}}_{\smallsetminus i})
}
\end{equation}
investigated in \cite{YiPa19}.
\end{example}

\begin{example}[minimization]
\label{ex:3}
Consider the setting of Problem~\ref{prob:1} where 
\ref{prob:1d} is replaced by
\begin{enumerate}[label={\rm[\alph*']}]
\setcounter{enumi}{1}
\item
For every $i\in I$, $\boldsymbol{f}_{\!i}=\boldsymbol{f}$,
where $\boldsymbol{f}\colon\KKK\to\RR$ is a differentiable convex 
function such that $\boldsymbol{Q}=\nabla\boldsymbol{f}$ is
Lipschitzian, 
\end{enumerate}
and, in addition, the following is satisfied:
\begin{enumerate}[label={\rm[\alph*]}]
\setcounter{enumi}{4}
\item
\label{prob:1e}
For every $k\in K$, $g_k\colon\GG_k\to\RR$ is G\^ateaux
differentiable.
\end{enumerate}
Then \eqref{e:nash} reduces to the multivariate minimization
problem 
\begin{equation}
\label{e:6887}
\minimize{\boldsymbol{x}\in\HHH}{
\sum_{i\in I}\big(\varphi_i(x_i)+\psi_i(x_i)\big)
+\boldsymbol{f}(\boldsymbol{M}\boldsymbol{x})
+\sum_{k\in K}(g_k+h_k)\Bigg(\sum_{j\in I}L_{k,j}x_j\Bigg)}.
\end{equation}
The only asynchronous block-iterative algorithm we know of to solve
\eqref{e:6887} is \cite[Algorithm~4.5]{Sadd20}, which is based on
different decomposition principles. Special cases of \eqref{e:6887}
are found in partial differential equations \cite{Atto08}, machine
learning \cite{Bric19}, and signal recovery \cite{Bric09}, where 
they were solved using synchronous and non block-iterative methods.
\end{example}

\end{document}